\documentclass[a4paper,11pt]{article}
\usepackage[utf8x]{inputenc}

\usepackage{a4wide}
\usepackage{amssymb}
\usepackage{amsmath}
\usepackage{amsthm}
\usepackage[mathscr]{euscript}
\usepackage[normalem]{ulem}
\usepackage[colorlinks=true]{hyperref}
\usepackage[usenames,dvipsnames]{xcolor}
\usepackage{todonotes}
\usepackage{appendix}
\usetikzlibrary{arrows,angles}

\numberwithin{equation}{section}

\theoremstyle{definition}

\newtheorem{definition}{Definition}[section]

\newtheorem{Theorem}[definition]{Theorem}
\newtheorem{Proposition}[definition]{Proposition}
\newtheorem{Lemma}[definition]{Lemma}
\newtheorem{Corollary}[definition]{Corollary}

\theoremstyle{remark}
\newtheorem{remark}[definition]{Remark}

\newcommand{\R}{\mathbb R}

\usepackage[xcolor]{changebar}
\cbcolor{blue}
\newcommand{\Ric}{\mathrm{Ric}}
\usepackage[inline]{enumitem}
\newcommand{\enumlabelformat}{\roman}
\newcommand{\enumlabelfont}[1]{#1}
\newlength{\thelabelsep}
\setlist{labelsep=\thelabelsep}
\setlist[enumerate,1]{font=\enumlabelfont,label=(\enumlabelformat*),leftmargin=2.5em}
\setlist[itemize]{leftmargin=2.5em,label=$-$}

\newcounter{inlineenum}
\renewcommand{\theinlineenum}{\enumlabelformat{inlineenum}}
\let\epsilon\varepsilon
\let\phi\varphi

\newcommand{\LpLS}{Lorentzian pre-length space }
\newcommand{\LpLSn}{Lorentzian pre-length space}
\newcommand{\ma}{\measuredangle}

\makeatletter
\let\save@mathaccent\mathaccent
\newcommand*\if@single[3]{%
  \setbox0\hbox{${\mathaccent"0362{#1}}^H$}%
  \setbox2\hbox{${\mathaccent"0362{\kern0pt#1}}^H$}%
  \ifdim\ht0=\ht2 #3\else #2\fi
  }
\newcommand*\rel@kern[1]{\kern#1\dimexpr\macc@kerna}
\newcommand*\widebar[1]{\@ifnextchar^{{\wide@bar{#1}{0}}}{\wide@bar{#1}{1}}}
\newcommand*\wide@bar[2]{\if@single{#1}{\wide@bar@{#1}{#2}{1}}{\wide@bar@{#1}{#2}{2}}}
\newcommand*\wide@bar@[3]{%
  \begingroup
  \def\mathaccent##1##2{%
    \let\mathaccent\save@mathaccent
    \if#32 \let\macc@nucleus\first@char \fi
    \setbox\z@\hbox{$\macc@style{\macc@nucleus}_{}$}%
    \setbox\tw@\hbox{$\macc@style{\macc@nucleus}{}_{}$}%
    \dimen@\wd\tw@
    \advance\dimen@-\wd\z@
    \divide\dimen@ 3
    \@tempdima\wd\tw@
    \advance\@tempdima-\scriptspace
    \divide\@tempdima 10
    \advance\dimen@-\@tempdima
    \ifdim\dimen@>\z@ \dimen@0pt\fi
    \rel@kern{0.6}\kern-\dimen@
    \if#31
      \overline{\rel@kern{-0.6}\kern\dimen@\macc@nucleus\rel@kern{0.4}\kern\dimen@}%
      \advance\dimen@0.4\dimexpr\macc@kerna
      \let\final@kern#2%
      \ifdim\dimen@<\z@ \let\final@kern1\fi
      \if\final@kern1 \kern-\dimen@\fi
    \else
      \overline{\rel@kern{-0.6}\kern\dimen@#1}%
    \fi
  }%
  \macc@depth\@ne
  \let\math@bgroup\@empty \let\math@egroup\macc@set@skewchar
  \mathsurround\z@ \frozen@everymath{\mathgroup\macc@group\relax}%
  \macc@set@skewchar\relax
  \let\mathaccentV\macc@nested@a
  \if#31
    \macc@nested@a\relax111{#1}%
  \else
    \def\gobble@till@marker##1\endmarker{}%
    \futurelet\first@char\gobble@till@marker#1\endmarker
    \ifcat\noexpand\first@char A\else
      \def\first@char{}%
    \fi
    \macc@nested@a\relax111{\first@char}%
  \fi
  \endgroup
}
\makeatother
\newcommand*{\bx}{\bar{x}}
\newcommand*{\by}{\bar{y}}
\newcommand*{\bz}{\bar{z}}

\newcommand*{\bp}{\bar{p}}

\title{The splitting theorem for globally hyperbolic Lorentzian length spaces with non-negative timelike curvature}

\author{Tobias Beran\footnote{Department of Mathematics, University of Vienna, Oskar-Morgenstern-Platz 1, 1090 Wien, Austria, \newline
tobias.beran@univie.ac.at, argam.ohanyan@univie.ac.at, felix.rott@univie.ac.at}, \\
Argam Ohanyan${}^*$,\\ %\footnotemark[\value{footnote}],
Felix Rott${}^*$,\\
Didier A.\ Solis\footnote{Facultad de Mathematicas, Universidad Aut\'onoma de Yucat\'an, Anillo Perif\'erico, Tablaje 13615, M\'erida, Mexico, didier.solis@correo.uady.mx},\\ %\footnotemark[\value{footnote}],
}

\begin{document}

\date{\today}

%\date{Received: date /Accepted: date}

\maketitle

\begin{abstract}
In this work, we prove a synthetic splitting theorem for globally hyperbolic Lorentzian length spaces with global non-negative timelike curvature containing a complete timelike line. Just like in the case of smooth spacetimes \cite{beem1985toponogov}, we construct complete, timelike asymptotes which, via triangle comparison, can be shown to fit together to give timelike lines. To get a control on their behaviour, we introduce the notion of parallelity of timelike lines in the spirit of the splitting theorem for Alexandrov spaces as proven in \cite{burago2001course} and show that asymptotic lines are all parallel. This helps to establish a splitting of a neighbourhood of the given line. We then show that this neighbourhood has the \textit{timelike completeness} property and is hence inextendible by a result in \cite{grant2019inextendibility}, which globalises the local result.
\vspace{1em}

\medskip


\end{abstract}
\newpage
\tableofcontents
\newpage

\section{Introduction}\label{sec:intro}

As originally formulated in the context of Riemannian geometry, splitting theorems are a class of results which establish that a Riemannian manifold $(M,g)$ subject to certain hypotheses on its curvature and its geodesic structure  actually \emph{splits}, i.e.\ it is isometric to a product. In most cases, the curvature assumption involves a weak inequality (usually on sectional or Ricci curvature) while the geodesic assumption is related to the existence of a \emph{line}, that is, a globally distance realising complete geodesic. Since lines are conjugate point free and positive curvature promotes the appearance of conjugate points, both features are expected to hold only under very special circumstances. 

The very first attempt to establish the metric product structure of a Riemannian manifold with non-negative curvature and having a line is attributed to Cohn-Vossen \cite{CohnVossen}, who proved that  a non-compact complete manifold of dimension 2 with sectional curvature $K\ge 0$ is either diffeomorphic to ${\mathbb{R}}^2$ or flat. However, the proof of this result relied on the Gauss-Bonnet Theorem and therefore was not suitable for generalisation to arbitrary dimensions.  The breakthrough that allowed such a generalisation was due to Toponogov, who in the mid 1960s proved his celebrated Triangle Comparison Theorem, which asserts that ``the angles of an arbitrary triangle in a Riemannian space made up of minimising arcs are no greater than the corresponding angles of the plane Euclidean triangle with sides of the same length" \cite{Toponogov03}. Armed with this tool, Toponogov established the archetype of a splitting theorem \cite{toponogov1964riemannian, Toponogov03}:

\begin{Theorem}\label{topothrm}
Let $(M, g)$ be a complete Riemannian manifold with $K \ge 0$ containing a line, then $(M , g)$ is isometric to $( {\mathbb{R}}^k\times N, g_0\oplus h)$ where $({\mathbb{R}}^k,g_0)$ is the standard Euclidean space and $(N, h)$ is complete and does not have any lines.
\end{Theorem}

Shortly after Toponogov's splitting theorem was published, Cheeger and Gromoll  made substantial progress by generalising it under the less stringent condition $\Ric\ge 0$ \cite{cheeger1971splitting}. Although their original motivation was rooted in the investigation of topological obstructions on manifolds of non-negative curvature \cite{CheegerGromoll02}, their result sparked the interest in the study of splitting theorems in many different contexts, like in the theory of orbifolds \cite{BorzellinoZhu}; in the study of curvature inequalities relating other tensors, (like the curvature operator \cite{Noronha} or Bakry-Emery tensor \cite{FLZ,Wang}). Furthermore, some more general versions of the Cheeger-Gromoll splitting theorem have been proven. For example, the curvature condition can be replaced by averaged Ricci inequalities \cite{Galloway05}, or by almost positivity of the Ricci tensor \cite{Cai02, Paeng}. There are also some local versions in which the splitting occurs either on tubular neighbourhoods or the complement of a compact set disjoint from a line \cite{Cai01, CaiGallowayLiu}.

Remarkably, Yau posed in 1982 the problem of establishing a Lorentzian analogue of the Cheeger-Gromoll splitting theorem \cite{Yau02}, thus inaugurating a very active field in which different versions of the Lorentzian splitting theorem were established \cite{beem1985toponogov,BEMG,eschenburg1988splitting,galloway1984splitting,galloway1989,newman1990proof}\footnote{Refer to chapter 14 in \cite{BeemEhrlich} for a detailed account.}, and some important problems have remained unsolved even to this day, like the famous Bartnik Conjecture \cite{bartnik1988remarks,GalBartnik}. In analogy to its Riemannian counterparts, the first Lorentzian splitting theorem was proven for globally hyperbolic spacetimes with non-positive sectional curvature on timelike planes. In \cite{beem1985toponogov,BEMG} Beem, Ehrlich, Markvorsen and Galloway proved what can be thought as the Lorentzian analogue of Toponogov's splitting theorem:
\begin{Theorem}\label{lorsplit2}
Let $(M,g)$ be a spacetime of dimension $n\ge 2$ that satisfies the
following conditions:
\begin{enumerate}
\item  $(M,g)$ is globally hyperbolic.
\item $K(\Pi )\le 0$ for all timelike planes $\Pi$.
\item $M$ has a timelike line.
\end{enumerate}
Then $(M,g)$ splits isometrically as $(M,g)\simeq (\mathbb{R}\times
N, -dt^2\oplus h)$, where $(N,h)$ is a complete Riemannian manifold.
\end{Theorem}
 Notice that global hyperbolicity can be seen as a more suitable analogue of Riemannian completeness than timelike geodesic completeness, since the latter fails to guarantee connectability by distance realising geodesic segments. The original proof of this result not only uses techniques related to Toponogov's theorem ---a Lorentzian triangle comparison due to Harris \cite{harris1982triangle}---, but also tools used in Cheeger-Gromoll's theorem, like Busemann functions and Hessian estimates.

Some of the key techniques used in Toponogov's original proof can be abstracted into an axiomatic setting rather than derived, thus expanding its range of applications to non-smooth contexts, particularly, to the realm of metric length spaces. Recall that a length space $(X,d)$ is a metric space whose distance is intrinsic, in other words, the distance $d(p,q)$ can be recovered as the infimum of the lengths of curves joining $p$ to $q$.  This is the setting for the so-called synthetic methods in geometry, that have proven instrumental in the development of recent mathematical landmarks such as the study of geometry in the large \cite{Gromov01,Gromov02}, precompactness theorems \cite{burago2001course,YGP}, geometric flows \cite{GigliPasq,AGS} and optimal transport \cite{Villani,Sturm}. An \emph{Alexandrov space with curvature bounded from below by $k$} ---or $\mathrm{Alex}(k)$ for short--- is a locally compact, complete and path connected (metric) length space $(X,d)$ on which the triangle comparison theorem holds \cite{burago2001course}. For $\mathrm{Alex}(0)$ spaces an analogue of Theorem \ref{topothrm} was first proved by Milka \cite{milka1967metric}. In this result, instead of the existence of a line, the existence of an $m$-affine function, (i.e. a function $g\colon M\to\mathbb{R}$ that when restricted to any unit speed geodesics satisfies the differential equation $g^{\prime\prime}+mg=0$) is  assumed. A warped product $I\times_g F$ with $g$ an $m$-affine function is called a \emph{cone} \cite{YGP}. Cones naturally have $m$-affine functions, and conversely, an $\mathrm{Alex}(k)$ space with a non-constant $m$-affine function splits as a cone, provided that a boundary condition is met \cite{AlexanderBishop, Mashiko}. In the particular case of a complete Riemannian manifold, Innami \cite{Innami} showed that  a $0$-affine function exists if and only if $M$ is isometric to a product $\mathbb{R}\times N$. Moreover, under the hypothesis of Theorem \ref{topothrm} the Busemann functions associated to a line are affine and thus Innami's theorem implies Toponogov's splitting theorem. Finally, let us note that in \cite{burago2001course} there is a proof of Milka's splitting theorem for Alexandrov spaces that resembles more closely the original works of Toponogov. The precise statement is as follows:

\begin{Theorem}\label{thrm:splitalex}
Let $(X,d)$ be an $\mathrm{Alex}(0)$ space containing a line. Then $X$ splits as a metric product $\mathbb{R}\times Y$, where $Y$ is an $\mathrm{Alex}(0)$ space.
\end{Theorem}

In their seminal work \cite{kunzinger2018lorentzian} Kunzinger and Sämann set the foundations for a synthetic approach to Lorentzian geometry. Their novel notion of Lorentzian (pre-)length space is suited to accommodate several different non-smooth scenarios such as cones \cite{alexander2019generalized,minguzzi2019causality}, spacetimes with $C^0$ metrics \cite{chrusciel2012lorentzian,ling01}, contact structures \cite{hendike} or causal boundaries \cite{ABS}. Moreover, in \cite{kunzinger2018lorentzian} the authors also introduced a notion of timelike curvature bounds in the same spirit of $\mathrm{Alex}(k)$ spaces.  Among the recent developments in this fast growing field we have detailed analyses of the causal structure \cite{hau2020causal,hevelin02}, extendibility \cite{grant2019inextendibility,galloway2018timelike}, convergence \cite{kunzinger2021null}, gluing techniques \cite{beran02, rott2022gluing} and the basis for a comparison theory \cite{barrera2022comparison,beran2022angles}.

In this work, we aim at proving a splitting theorem for Lorentzian length spaces in the spirit of Theorems \ref{topothrm}, \ref{lorsplit2} and \ref{thrm:splitalex}. In precise terms, we establish the following splitting result for Lorentzian length spaces:

\begin{Theorem}[Splitting]
\label{Theorem: Truesplitting}
Let $(X,d,\ll,\leq,\tau)$ be a connected, regularly localisable, globally hyperbolic Lorentzian length space with proper metric $d$ and global non-negative timelike curvature satisfying timelike geodesic prolongation and containing a complete timelike line $\gamma:\R \to X$. Then there is a $\tau-$ and $\leq$-preserving homeomorphism $f:\R\times S \to X$, where $S$ is a proper, strictly intrinsic metric space of Alexandrov curvature $\geq 0$.
\end{Theorem}

The paper is organised as follows: In Section \ref{sec:LLS}, we collect some well-known facts about Lorentzian (pre-)length spaces while focusing on (versions of) results best suited for our needs. We review the basic theory of Lorentzian (pre-)length spaces and discuss concepts like timelike curvature bounds, angles and extensions. Section \ref{sec:products} is dedicated to the study of product Lorentzian pre-length spaces of the form $\R \times X$, where $(X,d)$ is a metric space. Section \ref{sec:lines} is the heart of the proof and deals with the construction of co-rays and asymptotes, where we follow \cite{beem1985toponogov}. In order to have a control on the behaviour of lines formed as asymptotes to a given line, we introduce the notion of parallelity of timelike lines which is motivated by the splitting theorem for Alexandrov spaces as treated in \cite{burago2001course} and show that asymptotic lines are always timelike, of infinite $\tau$-length, and parallel to each other. In Section \ref{sec:proof} we first establish a local splitting result by endowing a cross section $S$ of the parallel asymptotes spanning $I(\gamma)$ with a natural metric and then showing that $I(\gamma)$ splits as $\R \times S$. Then we show that $I(\gamma)$ has the timelike completeness (TC) property, from which, via an inextendibility argument, $I(\gamma) = X$ follows. Finally, in Section \ref{sec: applicationsandoutlook} we note some classes of spacetimes which naturally satisfy the technical assumptions in our results, and hence split synthetically. We then give an outlook on open problems in the context of synthetic splitting theorems that may be addressed in future projects.

\subsection*{Notation and conventions}
\label{subsec: notationconventions}

Let us collect some notation and conventions that will be used throughout the paper.

$A \subset B$ means $A$ is a subset of $B$ (not necessarily a proper one). $\R^{1,1}$ is two dimensional Minkowski space with the Lorentzian metric $-dt^2 + dx^2$ and coordinates $(t,x)$. $\overline{\tau}$ denotes the time separation on $\R^{1,1}$. A \emph{proper} metric space $(X,d)$ is a metric space such that all closed balls are compact. A metric space is called \emph{length space} or \emph{intrinsic} if the distance between two points is the infimum of lengths of curves running between them, and \emph{strictly intrinsic} if that infimum is a minimum, i.e.\ between any two points there exists a distance-realising curve.

We denote an open ball with radius $R$ around a point $x$ in a metric space by $B_R(x)$, and the corresponding closed ball by $\bar{B}_R(x)$ (not to be confused with the closure $\overline{B_R(x)}$ of the open ball).

\section{Basic theory of Lorentzian (pre-)length spaces}\label{sec:LLS}
\subsection{Lorentzian (pre-)length spaces}\label{subsec:LLS:LpLS}

We give a brief review of the theory of Lorentzian (pre-)length spaces. We focus on the specific results that we need and give proofs if they cannot be found in the literature, otherwise we give precise references. In particular, we refer to \cite{kunzinger2018lorentzian} for a detailed treatment. We start with a few basic definitions.

\begin{definition}[Lorentzian pre-length spaces]\par
\label{definition: LpLS}
A \emph{Lorentzian pre-length space} $(X,d,\ll,\leq,\tau)$ consists of a metric space $(X,d)$, a reflexive and transitive relation $\leq$ (the \textit{causal relation}), a transitive relation $\ll$ (the \textit{timelike relation}) contained in $\leq$, and a lower semi-continuous map (\textit{time separation}) $\tau:X \times X \to [0,\infty]$ with the following properties: $\tau(x,y)=0$ if $x \not\leq y$, and $\tau(x,y) > 0$ if and only if $x \ll y$. Moreover, if $x \leq y \leq z$, then the following \emph{reverse triangle inequality} holds:
\begin{align*}
    \tau(x,z) \geq \tau(x,y) + \tau(y,z).
\end{align*}
\end{definition}

In the synthetic theory, we employ the usual nomenclature and notation well-known from Lorentzian geometry, such as $I^{\pm}(x)$, $J^{\pm}(x)$ for timelike and causal pasts and futures, as well as $I(x,y):=I^+(x) \cap I^-(y)$ and $J(x,y):=J^+(x) \cap J^-(y)$ for timelike and causal diamonds, respectively.

\begin{Lemma}[Openness of timelike futures and push-up]\par
\label{Lemma: pushupandopennness}
Let $(X,d,\ll,\leq,\tau)$ be a Lorentzian pre-length space and $x \ll y \leq z$ or $x \leq y \ll z$. Then $x \ll z$. In particular, for any $x \in X$, the sets $I^{\pm}(x)$ are open. Moreover, the relation $\ll$ is open in $X \times X$.
\begin{proof}
See \cite[Lem.\ 2.10, Lem.\ 2.12, Prop.\ 2.13]{kunzinger2018lorentzian}.
\end{proof}
\end{Lemma}

\begin{definition}[Causal curves]\par 
\label{Definition: causalcurves}
Let $(X,d,\ll,\leq,\tau)$ be a Lorentzian pre-length space. A non-constant curve $\gamma:I \to X$, $I$ an interval, is called \emph{future directed causal} (resp.\ \emph{future directed timelike}) if it is locally Lipschitz and for all $t_1,t_2 \in I$ with $t_1 < t_2$ we have that $\gamma(t_1) \leq \gamma(t_2)$ (resp.\ $\gamma(t_1) \ll \gamma(t_2)$). It is called \emph{future directed null} if it is future directed causal and no two points on the curve are $\ll$-related. \emph{Past directed causal/timelike/null} curves are defined dually. From now on, unless explicitly stated otherwise, we assume all causal curves to be future directed.
\end{definition}

We will sometimes deal with causal curves which are not locally Lipschitz. For these curves, some important results (e.g.\ the limit curve theorem) will not hold in general, however, many important properties (especially those that are purely causal-theoretic in nature) will continue to be true.

\begin{definition}[Continuous causal curves]\par 
\label{definition: continuouscausalcurves}
Let $(X,d,\ll,\leq,\tau)$ be a Lorentzian pre-length space. A non-constant curve $\gamma:I \to X$, $I$ an interval, is called a \emph{future directed continuous causal curve} (resp.\ \emph{future directed continuous timelike curve}) if it is continuous and for all $t_1,t_2 \in I$ with $t_1 < t_2$ we have that $\gamma(t_1) \leq \gamma(t_2)$ (resp.\ $\gamma(t_1) \ll \gamma(t_2)$). It is called a \emph{future directed continuous null curve} if it is a future directed continuous causal curve and no two points on the curve are $\ll$-related. Past versions of these notions are defined dually. From now on, unless explicitly stated otherwise, we assume all continuous causal curves to be future directed.
\end{definition}

We emphasise that a \emph{causal/timelike/null curve} is always understood to be locally Lipschitz. If \emph{continuous} causal/timelike/null curves are meant, we will state that explicitly. Wherever possible, we will state definitions and results for continuous causal curves instead of (locally Lipschitz) causal curves for greater generality. In most cases, the proofs given in the literature for these results apply word for word to the continuous case.

\begin{definition}[Extensions of causal curves]
\label{definition: extensions of causal curves}
Let $(X,d,\ll,\leq,\tau)$ be a Lorentzian pre-length space and let $\gamma:I \to X$ be a continuous causal curve, where $I$ is any interval. We say that $\gamma$ is \emph{(continuously) extendible} if there exists a (continuous) causal curve $\tilde{\gamma}:J \to X$, $J$ an interval with $J \supsetneqq I$, such that $\tilde{\gamma}|_{I} = \gamma$. Otherwise, we say $\gamma$ is \emph{(continuously) inextendible}. We refer to (in)extendibility to points in the future resp.\ past directions of $\gamma$ as \textit{future/past (in)extendibility}.\footnote{Geodesics defined on intervals $[a,b]$ with $a = -\infty$ or $b=\infty$ should be thought of as properly reparametrised.}
\end{definition}

\begin{definition}[$\tau$-length]\par 
\label{definition: taulength}
Let $\gamma:[a,b] \to X$ be a future directed continuous causal curve. 
Then its \emph{$\tau$-length} is defined by
\begin{align*}
    L_{\tau}(\gamma):=\inf\left\{\sum_{i=0}^{N-1} \tau(\gamma(t_i),\gamma(t_{i+1})) \mid a = t_0 < t_1 < \dots < t_N = b \right\}. 
\end{align*}
If $\gamma$ is past directed causal, then $L_{\tau}(\gamma)$ is defined analogously. For continuous causal curves defined on half-open intervals, e.g. $[a,b)$, one takes the limit of $L_{\tau}(\gamma|_{[a,c]})$ as $c \to b$, and similarly in the other cases.
\end{definition}

The $\tau$-length of a continuous causal curve is invariant under reparametrisations. We will have more to say on these topics later.

\begin{definition}[Maximising causal curves]\par 
\label{definition: maximalcurve}
A future directed (continuous) causal curve $\gamma:[a,b] \to X$ is called \emph{maximising} (or $\tau$-\emph{maximising}) if $\tau(\gamma(a),\gamma(b)) = L_{\tau}(\gamma)$, and analogously for past directed (continuous) causal curves. We also refer to such curves as \emph{(continuous) distance realisers}.
\end{definition}

Lorentzian pre-length spaces where between each pair of causally related points there is a maximising (locally Lipschitz) causal curve are referred to as \emph{strictly intrinsic} or \emph{geodesic}.

Next, we define some of the steps on the causal ladder for Lorentzian pre-length spaces.

\begin{definition}[Causality conditions]\par 
A Lorentzian pre-length space $(X,d,\ll,\leq,\tau)$ is called
\begin{enumerate}
    \item \emph{non-totally imprisoning} if for every compact set $K\subset X$ there is $C > 0$ such that the $d$-lengths of all causal curves in $K$ is bounded by $C$,
    \item \emph{strongly causal} if the Alexandrov topology generated by the subbasis of timelike diamonds $\{I(x,y) \mid x,y \in X\}$ coincides with the metric topology,
    \item \emph{globally hyperbolic} if it is non-totally imprisoning and all \emph{causal diamonds} $J(x,y)$ are compact.
\end{enumerate}
\end{definition}

\begin{Lemma}
\label{lemma: stronglycausalcausallyconvexbase}
Let $(X,d,\ll,\leq,\tau)$ be a strongly causal Lorentzian pre-length space. Then each $x \in X$ has a neighbourhood basis consisting of causally convex\footnote{A set $V \subset X$ is called \emph{causally convex} if for all $p,q \in V$ we have $J(p,q) \subset V$.} neighbourhoods. Even more is true: $X$ has a topological basis of causally convex neighbourhoods.
\begin{proof}
This is trivial since by definition finite intersections of timelike diamonds are a basis for the topology, but these sets are causally convex.
\end{proof}
\end{Lemma}

\begin{definition}[Causal path-connectedness]\par
\label{definition: causalpathconnect}
A Lorentzian pre-length space $X$ is called \emph{causally path-connected} if for all $x < y$ there is a causal curve from $x$ to $y$ and for all $x \ll y$ there is a timelike curve from $x$ to $y$.
\end{definition}

For the next definition, we use the version given in \cite[Def.\ 2.16]{hau2020causal} rather than \cite[Def.\ 3.4]{kunzinger2018lorentzian} as it more closely resembles the case of smooth spacetimes. Before doing so, note that on causally path-connected Lorentzian pre-length spaces $X$, we can introduce a \emph{local causal relation}: Let $U \subset X$ be open and $x,y \in U$, then we say $x \leq_U y$ if there exists a future causal curve from $x$ to $y$ entirely contained in $U$.

\begin{definition}[Local causal closedness]\par 
\label{definition: loccausalclosed}
Let $X$ be causally path-connected and $U \subset X$ be open. We say that $U$ is \emph{causally closed} if for all sequences $p_n,q_n$ in $U$ converging to $p,q \in U$ and $p_n \leq_U q_n$, we have that $p \leq_U q$. We say that $X$ is \emph{locally causally closed} if every point in $X$ has a causally closed neighbourhood.
\end{definition}

\begin{definition}[Global causal closedness]\par 
\label{definition: globcausalclosed}
A Lorentzian pre-length space $X$ is called \textit{globally causally closed} if $\leq$ is a closed relation.
\end{definition}

\begin{definition}[$d$-compatibility]\par 
A Lorentzian pre-length space $(X,d,\ll,\leq,\tau)$ is called $d$-\emph{compatible} if for every $x \in X$ there is a neighbourhood $U$ of $x$ and a constant $C > 0$ such that the $d$-lengths of all causal curves contained in $U$ are bounded above by $C$.
\end{definition}

\begin{Lemma}
\label{Lemma: Nontotalimprcausalcurvesleavecompsets}
Let $(X,d,\ll,\leq,\tau)$ be a locally causally closed, $d$-compatible, globally hyperbolic Lorentzian pre-length space. Then no compact set in $X$ contains an inextendible causal curve.
\begin{proof}
See \cite[Cor.\ 3.15]{kunzinger2018lorentzian}.
\end{proof}
\end{Lemma}

\begin{definition}[Localisability]\par 
\label{definition: localizablity}
A Lorentzian pre-length space $(X,d,\ll,\leq,\tau)$ is called \emph{localisable} if for each $x \in X$ there exists a neighbourhood (called \emph{localising neighbourhood}) $\Omega_x$ containing $x$ with the following properties: 
\begin{enumerate}
    \item Causal curves in $\Omega_x$ have uniformly bounded $d$-length, i.e., $\Omega_x$ is $d$-compatible.
    \item For each $y \in \Omega_x$, $I^{\pm}(y) \cap \Omega_x \neq \emptyset$.
    \item There is a continuous map (\emph{local time separation}) $\omega_x:\Omega_x \times \Omega_x \to [0,\infty)$ such that $\Omega_x$ is a Lorentzian pre-length space upon restricting $d,\ll,\leq$.
    \item For all $p,q \in \Omega_x$ with $p < q$ there exists a causal curve $\gamma_{pq}$ from $p$ to $q$ entirely in $\Omega_x$ with maximal $\tau$-length among all causal curves from $p$ to $q$ contained in $\Omega_x$, as well as $L_{\tau}(\gamma_{pq}) = \omega_x(p,q)$.
\end{enumerate}
Note that this necessarily means that $\omega_x$ is of the form 
\begin{equation}
\omega_x(p,q):=\max \{ L_{\tau}(\gamma) \mid \gamma \text{ is a causal curve from $p$ to $q$ in $\Omega_x$.} \}.  
\end{equation}
If in addition, for $p,q \in \Omega_x$ with $p\ll q$ the curve $\gamma_{pq}$ is timelike and strictly longer than any causal curve from $p$ to $q$ in $\Omega_x$ containing a null segment, then $\Omega_x$ is called a \emph{regular localising neighbourhood}, and $X$ \emph{regularly localisable} if any point has a regular localising neighbourhood. Moreover, $X$ is called \emph{strongly localisable} if each point has a neighbourhood base of localising neighbourhoods and \emph{SR-localisable} if each point has a neighbourhood base of regular localising neighbourhoods.
\end{definition}

\begin{Lemma}
\label{Lemma: locimpliesstronglylocforstronglycausal}
Let $(X,d,\ll,\leq,\tau)$ be a strongly causal Lorentzian pre-length space.
\begin{enumerate}
    \item If $X$ is localisable, then it is strongly localisable.
    \item If $X$ is regularly localisable, then it is SR-localisable.
\end{enumerate}
\begin{proof}
By Lemma \ref{lemma: stronglycausalcausallyconvexbase}, each point in $X$ has a neighbourhood basis of causally convex neighbourhoods. Now suppose $X$ is localisable, take $x \in X$ and let $\Omega_x$ be a localising neighbourhood of $x$ and let $U \subset \Omega_x$ be any neighbourhood of $x$. Then there is a causally convex neighbourhood $V$ of $x$ such that $V \subset U$. We are done if we show that $V$ is a localising neighbourhood. But this is easy to see since any causal curve starting and ending in $V$ is entirely contained in there. If $X$ is regularly localisable, then $V$ is even a regular localising neighbourhood.
\end{proof}

\end{Lemma}

Regularly localisable spaces have the following important property:

\begin{Theorem}[Causal character of maximising causal curves]\par
\label{theorem: causalcharofmaximizers}
In a regularly localisable Lorentzian pre-length space, maximising causal curves are either timelike or null.
\begin{proof}
See \cite[Thm.\ 3.18]{kunzinger2018lorentzian}.
\end{proof}
\end{Theorem}

Note that in localisable Lorentzian pre-length spaces, the sets $I^{\pm}(y)$ are never empty for any $y \in X$.

\begin{definition}[Geodesic]\par 
\label{definition: geodesic}
Let $X$ be a localisable Lorentzian pre-length space and $\gamma:I \to X$ a (continuous) causal curve. Then $\gamma$ is a (future-directed causal) \emph{(continuous) geodesic} if for each $t_0 \in I$ there exists a localising neighbourhood $\Omega$ of $\gamma(t_0)$ and a neighbourhood  $[c,d] \subset I$ of $t_0$ such that $\gamma([c,d]) \subset \Omega$ and $\gamma|_{[c,d]}$ is maximising in $\Omega$ with respect to the local time separation. 
\end{definition}

We will sometimes need the notion of \emph{geodesic (in)extendibility} which is defined in terms of the existence of a geodesic extending a given geodesic. Note that inextendibility implies geodesic inextendibility but the converse need not hold in general.

For many arguments in the proof of the splitting theorem, the property that geodesics can always be extended to open domains will be essential. While this is a consequence of the ODE theory of the geodesic equation in the case of spacetimes, we will need to assume it here:

\begin{definition}[Timelike geodesic prolongation]
\label{definition: geodesicprolongation}
A localisable \LpLS $X$ is said to have the \emph{timelike geodesic prolongation} property if any maximising timelike segment $\gamma:[a,b]\to X$ can be extended to a timelike geodesic with an open domain, i.e.\ there is an $\varepsilon>0$ and a timelike geodesic $\bar{\gamma}:(a-\varepsilon,b+\varepsilon)\to X$ with $\bar{\gamma}|_{[a,b]}=\gamma$.
\end{definition}

It is clear that this is satisfied in smooth strongly causal spacetimes considered as a \LpLSn. One easily sees that spacetimes with timelike boundary do not satisfy this.

\begin{definition}[Locally quasiuniformly maximising]
\label{definition: locuniformmaxspace}
A localisable Lorentzian pre-length space $X$ is called \emph{locally quasiuniformly maximising} if for each $x \in X$ there exists a localisable neighbourhood $U$ containing $x$ such that each causal geodesic entirely contained in $U$ is maximising in $U$ with respect to the local time separation.
\end{definition}

\begin{Proposition}[Upper semicontinuity of Lorentzian arclength]\par
\label{Proposition: uscarclength}
Let $(X,d,\ll,\leq,\tau)$ be a strongly causal and localisable Lorentzian pre-length space. Then $L_{\tau}$ is upper semi-continuous, i.e.\ if $\gamma_n:[a,b] \to X$ are causal curves converging $d$-uniformly to a causal curve $\gamma:[a,b] \to X$, then
\begin{align*}
    L_{\tau}(\gamma) \geq \limsup_{n \to \infty} L_{\tau}(\gamma_n).
\end{align*}
\begin{proof}
See \cite[Prop.\ 3.17]{kunzinger2018lorentzian}.
\end{proof}
\end{Proposition}

\begin{Proposition}
\label{Proposition: extgeovsextcont}
Let $X$ be a strongly causal, localisable, locally quasiuniformly maximising Lorentzian pre-length space and let $\gamma:[a,b) \to X$ be a causal geodesic. Then $\gamma$ is geodesically extendible to $[a,b]$ if and only if it is extendible as a continuous curve to $[a,b]$.
\begin{proof}
See \cite[Prop.\ 4.6]{grant2019inextendibility} and note that the proof really requires the space to be locally quasiuniformly maximising.
\end{proof}

\end{Proposition}

We now state the limit curve theorem in the case of strongly causal Lorentzian pre-length spaces, which will be sufficient for our needs. Note that in this case, Definition \ref{definition: loccausalclosed} is equivalent to \cite[Def.\ 3.4]{kunzinger2018lorentzian}.

\begin{Theorem}[Limit curve theorem]\par
\label{Theorem: LimitCurveTheorem}
Let $(X,d,\ll,\leq,\tau)$ be a locally causally closed, localisable, strongly causal, $d$-compatible Lorentzian pre-length space with proper metric $d$. Let $\gamma_n:[0,L_n] \to X$ be a sequence of causal curves parametrised by $d$-arclength, and suppose that $L_n = L^d(\gamma_n) \to \infty$. If the sequence $\gamma_n(0)$ has an accumulation point $x \in X$, then some subsequence of the $\gamma_n$ converges locally uniformly to a future inextendible causal curve $\gamma:[0,\infty) \to X$. Moreover, if the $\gamma_n$ are maximising, then so is $\gamma$. In this case, for any $T > 0$, $L_{\tau}(\gamma|_{[0,T]}) = \lim_n L_{\tau}(\gamma_n|_{[0,T]})$.
\begin{proof}
The statement about the existence of the limit curve is shown in \cite[Thm.\ 3.14]{kunzinger2018lorentzian}. Now suppose that $\gamma_n$ are maximising and w.l.o.g.\ let $\gamma_n$ itself converge locally uniformly to $\gamma$. Consider any $T > 0$. Then by upper semicontinuity of $L_{\tau}$ and lower semicontinuity of $\tau$,
\begin{align*}
    \limsup_n L_{\tau}(\gamma_n|_{[0,T]}) \leq L_{\tau}(\gamma|_{[0,T]}) \leq \tau(\gamma(0),\gamma(T)) \leq \liminf_n \tau(\gamma_n(0),\gamma_n(T)) = \liminf_n L_{\tau}(\gamma_n|_{[0,T]}).
\end{align*}
Thus, equality holds everywhere. In particular, $\gamma$ is maximising and $\lim_n L_{\tau}(\gamma_n|_{[0,T]}) = L_{\tau}(\gamma|_{[0,T]})$.
\end{proof}
\end{Theorem}

In the context of the splitting theorem, we are interested in Lorentzian pre-length spaces with significantly more structure, which we will discuss now.

\begin{definition}[Lorentzian length spaces]\par
\label{definition: LLS}
A locally causally closed, causally path-connected and localisable Lorentzian pre-length space $(X,d,\ll,\leq,\tau)$ is called \emph{Lorentzian length space} if for all $x \leq y$, $x \neq y$,
\begin{align*}
    \tau(x,y) = \sup\{L_{\tau}(\gamma) \mid \gamma \text{ future directed causal from }x \text{ to }y\}.
\end{align*}
\end{definition}

Note that since by definition $L_{\tau}(\gamma) \leq \tau(p,q)$ for any (even continuous) causal curve from $p$ to $q$, one can equivalently take the supremum over continuous causal curves to obtain the time separation in the case of Lorentzian length spaces.

\begin{Proposition}[Causality of Lorentzian length spaces]\par
\label{Proposition: StrongcausalityGlobhypLLS}
\ \begin{enumerate}
    \item A Lorentzian length space is strongly causal if and only if each point has a neighbourhood base of causally convex neighbourhoods.
    \item Strongly causal Lorentzian length spaces are non-totally imprisoning.
    \item Globally hyperbolic Lorentzian length spaces are strongly causal. Moreover, $\tau$ is continuous and finite and for any two causally related points there exists a maximising causal curve connecting them.
\end{enumerate}
\begin{proof}
See \cite[Thm.\ 3.26, Thm.\ 3.28, Thm.\ 3.30]{kunzinger2018lorentzian}.
\end{proof}
\end{Proposition}

\begin{Lemma}
\label{Lem:strongly_strongly_causal}
Let $(X,d,\ll,\leq,\tau)$ be a strongly causal Lorentzian pre-length space such that for each $x \in X$ there exists a timelike curve $\gamma:[-\varepsilon,\varepsilon] \to X$ with $\gamma(0) = x$. Then $\{I(x,y) \mid x,y \in X\}$ is a base for the topology on $X$. Moreover, for each $x \in X$, each timelike curve $\gamma:[-\varepsilon,\varepsilon] \to X$ with $\gamma(0)=x$ and parameters $0<s_n,t_n \to 0$, $\{I(\gamma(-s_n),\gamma(t_n)) \mid n \in \mathbb{N}\}$ is a neighbourhood base at $x$.
\begin{proof}
Let $x \in X$ with $x \in U$, $U \subset X$ open. By the definition of the Alexandrov topology, finite intersections of the sets $I^+(y) \cap I^-(z)$ with $y,z \in X$ are a base. Thus, there exist $x_1,\dots,x_n,y_1,\dots,y_n \in X$ such that 
\begin{align*}
    x \in V:=I^+(x_1) \cap I^-(y_1) \cap \dots \cap I^+(x_n) \cap I^-(y_n) \subset U.
\end{align*}
Now let $\gamma:[-\varepsilon,\varepsilon] \to X$ be a timelike curve with $\gamma(0) = x$. By continuity, maybe after making $\varepsilon$ smaller, we may assume that $\gamma([-\varepsilon,\varepsilon]) \subset V$. Then $W:=I^+(\gamma(-\varepsilon)) \cap I^-(\gamma(\varepsilon))$ is an open set containing $x$ and we claim that $W \subset V$. To see this, let $z \in W$, then for any $i = 1,\dots,n$ we have $x_i \ll \gamma(-\varepsilon) \ll z$, hence $z \in I^+(x_i)$. Similarly, $z \ll \gamma(\varepsilon) \ll y_i$, hence $z \in I^-(y_i)$. Thus $z \in V$. This proves that $\{I(x,y) \mid x,y \in X\}$ is a base for the topology of $X$, and the second claim then easily follows from the arguments above.
\end{proof}
\end{Lemma}

Lorentzian pre-length spaces where each point lies in the interior of a timelike curve are future and past approximating in the sense of \cite[Def.\ 2.17]{burtscher2021time} (see also \cite[Lem.\ 2.18]{burtscher2021time})

\begin{Lemma}[Geodesics in strongly causal Lorentzian length spaces]\par 
Let $(X,d,\ll,\leq,\tau)$ be a strongly causal Lorentzian length space. A (continuous) causal curve $\gamma:I \to X$ is a (continuous) geodesic if and only if for each $t_0 \in I$ there exists a subinterval $[c,d] \subset I$ with $t_0 \in [c,d]$ such that $\gamma|_{[c,d]}$ is $\tau$-maximising.
\begin{proof}
This follows from \cite[Lem.\ 4.3]{grant2019inextendibility}.
\end{proof}

\end{Lemma}

To conclude this subsection, we discuss $d$-arclength and $\tau$-arclength reparametrisations of causal curves. 

\begin{remark}[On $d$-arclength parametrisations]\par
\label{remark: darclengthparametrizations}
Let $(X,d,\ll,\leq,\tau)$ be a globally hyperbolic Lorentzian length space with proper metric $d$ and $\gamma:[a,b) \to X$ a causal curve that is inextendible at $b$. Then $\gamma$ has infinite $d$-length: By Lemma \ref{Lemma: Nontotalimprcausalcurvesleavecompsets}, $\gamma$ has to leave every compact set, in particular it leaves each compact ball $\overline{B}_n(\gamma(a))$, which shows that $L^d(\gamma) = \infty$. If we reparametrise $\gamma$ by $d$-arclength (note that this is in general not possible for continuous causal curves), it is thus defined on $[0,\infty)$. Similarly, if we have a doubly inextendible causal curve $\tilde{\gamma}:(a,b) \to X$, then upon reparametrising it by $d$-arclength, it is defined on $\R$, so that $L^d(\tilde{\gamma}|_{[c,d]}) = |d-c|$ (see also \cite[Rem.\ 2.5]{burtscher2021time}). Note that $d$-arclength parametrisations are globally $1$-Lipschitz.
\end{remark}

\begin{Lemma}[$\tau$-arclength parametrisations for inextendible curves]\par
\label{Lemma: tauarclengthparametrizations}
Let $(X,d,\ll,\leq,\tau)$ be a Lorentzian pre-length space with continuous time separation $\tau$ which satisfies $\tau(x,x) = 0$ for all $x \in X$. Let $\gamma:[a,b) \to X$ be a future inextendible continuous timelike curve defined on a half-open interval with $L_{\tau}(\gamma) \leq \infty$ and $L_{\tau}(\gamma|_{[a,c]}) < \infty$ for each $c\in [a,b)$. Then the map $\phi:[a,b) \to [0,L_{\tau}(\gamma))$, $t \mapsto L_{\tau}(\gamma|_{[a,t]})$ is continuous and strictly increasing. Moreover, $\tilde{\gamma}:=\gamma \circ \phi^{-1}:[0,\infty) \to X$ is a weak reparametrisation of $\gamma$ with respect to $\tau$-arclength. Similarly, a doubly inextendible continuous timelike curve $\gamma:(a,b) \to X$ can be weakly parametrised by $\tau$-arclength and is then defined on $\left(-L_{\tau}(\gamma|_{(a,t_0]}),L_{\tau}(\gamma|_{[t_0,b)})\right)$, where $t_0 \in (a,b)$ is arbitrary. 
\begin{proof}
This can be seen by applying \cite[Lem.\ 3.33, Lem.\ 3.34]{kunzinger2018lorentzian} to each $\gamma|_{[a,\tilde{b}]}$ with $\tilde{b} < \infty$.
\end{proof}
\end{Lemma}

Note that, in general, the $\tau$-arclength parametrisation results in a continuous timelike curve, even if the original curve is locally Lipschitz.

In the following, whenever $\tau$-arclength parametrised curves are mentioned, it is always implicitly understood that those curves have finite $\tau$-length on compact subintervals. In localisable spaces this is anyway the case.

\subsection{Timelike curvature bounds}\label{subsec:LLS:tlCurvBds}

Establishing an analogue to sectional curvature bounds in semi-Riemannian manifolds à la \cite{alexander2008triangle} was one of the main goals in the first work on Lorentzian length spaces \cite{kunzinger2018lorentzian}. 
Historically, splitting theorems always used some type of curvature condition, either Ricci or sectional curvature.
By $M_k$ we denote the Lorentzian model space of constant sectional curvature $k$ and, in this subsection only, $\bar{\tau}$ denotes the time separation in $M_k$ (in general $\bar{\tau}$ will only be used as the time separation in $M_0 = \R^{1,1}$). For more details on this, see \cite[Definition 4.5]{kunzinger2018lorentzian}. We will always assume that all mentioned triangles satisfy size bounds for $M_k$, cf. \cite[Lemma 4.6]{kunzinger2018lorentzian}. However, since we are mainly working with $k=0$, we do not have to worry about this.

We call three points $x_1,x_2,x_3 \in X$ that are timelike related together with maximisers between them a \emph{timelike triangle}. We denote such triangles by $\Delta(x_1,x_2,x_3)$. 

\begin{definition}
Let $X$ be a \LpLS and $\Delta(x_1,x_2,x_3)$ be a timelike triangle. Then a timelike triangle $\Delta(\bar{x}_1,\bar{x}_2,\bar{x}_3)$ in $M_k$ such that $\tau(x_i,x_j)=\bar{\tau}(\bar{x}_i,\bar{x}_j)$ is called a \emph{comparison triangle} for $\Delta(x_1,x_2,x_3)$. It always exists (if the size bounds are satisfied) and is unique up to isometries of $M_k$.
\end{definition}

\begin{definition}[Timelike curvature bounds]
\label{Definition: local tl curv bound}
A Lorentzian pre-length space $X$ is said to satisfy a \emph{timelike curvature bound from below by $k$}, if each point in $X$ has a so-called \text{comparison neighbourhood} $U$, satisfying the following:
\begin{enumerate}
    \item $\tau|_{U \times U}$ is finite and continuous.
    
    \item For all $x,y \in U$ with $x \ll y$ there exists a causal maximiser $\gamma$ from $x$ to $y$ which is entirely contained in $U$.
    
    \item Let $\Delta(x,y,z)$ be a triangle in $U$ and let $\Delta(\bx,\by,\bz)$ be its comparison triangle in $M_k$. Let $p,q \in \Delta(x,y,z)$ and let $\bar{p},\bar{q}$ be the corresponding comparison points in $\Delta(\bar{x},\bar{y},\bar{z})$. Then $\tau(p,q) \leq \bar{\tau}(\bar{p},\bar{q})$. 
\end{enumerate}
For curvature bounds from above, the inequality in (iii) is reversed, but we will not need these types of curvature bounds. If $k=0$, we also say that $X$ has \emph{non-negative timelike curvature} resp.\ \emph{non-positive timelike curvature}.
\end{definition}

Evidently, curvature bounds are only formulated locally as of yet. There is, however, a very natural way of globalising this concept. 

\begin{definition}[Global curvature bound]
\label{definition: globalcurvbounds}
A Lorentzian pre-length space $X$ is said to have \emph{global curvature bounded below by $k$} if $X$ itself is a comparison neighbourhood in the sense of Definition \ref{Definition: local tl curv bound}.
\end{definition}

\begin{remark}[Continuous triangles vs.\ Lipschitz triangles]
\label{remark: contvsLipschitztriangles}
If $X$ is a globally hyperbolic Lorentzian length space with global nonnegative timelike curvature, then in fact curvature comparison even holds for timelike triangles where the maximisers are continuous: Indeed, suppose $\Delta:=\Delta_{C^0}(x,y,z)$ is a continuous timelike triangle, and let $p,q \in \Delta$. Due to the properties of $X$, we find (Lipschitz) maximisers from the endpoints of $\Delta$ to $p,q$, respectively. The concatenations at $p$ resp.\ $q$ of those maximisers are again maximisers because the sides on $\Delta$ are (continuous) maximisers. Hence we have realised $p,q$ on a Lipschitz triangle. 
\end{remark}

One of the most commonly used implications of lower (timelike) curvature bounds is the prohibition of branching of distance-realisers. A formulation of this result for Lorentzian pre-length spaces was first given in \cite[Theorem 4.12]{kunzinger2018lorentzian}. However, with the introduction of hyperbolic angles in \cite{beran2022angles} it was possible to generalise this result by omitting some of the additional assumptions:

\begin{Theorem}[Timelike non-branching]\label{thm:non-branching}
Let $X$ be a strongly causal Lorentzian pre-length space with timelike curvature bounded below by $k \in \R$. Then timelike distance realisers cannot branch, i.e., if $\alpha, \beta: [-\varepsilon, \varepsilon] \to X$ are timelike distance realisers such that there exists $t_0 \in \R$ with $\alpha|_{[-\varepsilon,t_0]}=\beta|_{[-\varepsilon,t_0]}$, then $\alpha=\beta$.
\end{Theorem}
\begin{proof}
See \cite[Theorem 4.7]{beran2022angles}.
\end{proof}

The non-branching of timelike distance realisers is a key property of spaces with lower curvature bounds and will appear in various forms in our proof of the splitting theorem.

\subsection{Angles and comparison angles}\label{subsec:LLS:angles}

Hyperbolic angles in Lorentzian pre-length spaces were introduced in \cite{beran2022angles} and \cite{barrera2022comparison}, where the latter puts a bigger focus on comparison results. We will follow the conventions of the former reference.

\begin{Lemma}[The law of cosines ($k=0$)]\par
\label{Lemma: hyperboliclawofcosines}
Let $X=\R^{1,1}$ be Minkowski space and $x_1,x_2,x_3$ be three points which are timelike related (an (unordered) timelike triangle). Let $a_{ij}=\max(\tau(x_i,x_j),\tau(x_j,x_i))$ (note one of these is zero anyway). Let $\omega$ be the hyperbolic angle between the straight lines $x_1x_2$ and $x_2x_3$ at $x_2$. Set $\sigma=1$ if $x_2$ is not a time endpoint of the triangle (i.e., $x_1\ll x_2\ll x_3$ or $x_3\ll x_2\ll x_1$) and $\sigma=-1$ if $x_2$ is a time endpoint of the triangle (i.e., $x_2\ll x_1,x_3$ or $x_1,x_3\ll x_2$). Then we have:
\begin{equation*}
a_{13}^2=a_{12}^2+a_{23}^2 +\sigma a_{12}a_{23}\cosh(\omega).
\end{equation*}
In particular, when only changing one side-length, the angle $\omega$ is a monotonously increasing function of the longest side-length and monotonously decreasing in the other side-lengths.
\end{Lemma}
\begin{proof}
See \cite[Appendix A]{beran2022angles}.
\end{proof}

Note that for $a_{ij}>0$ satisfying a reverse triangle inequality and choosing an appropriate $\sigma=\pm 1$, we can always solve this equation for $\omega$.

\begin{definition}[Comparison angles]
Let $X$ be a \LpLS and $x_1,x_2,x_3$ three timelike related points. Let $\bar{x}_1,\bar{x}_2,\bar{x}_3\in\R^{1,1}$ be a comparison triangle for $x_1,x_2,x_3$. We define the \emph{comparison angle} $\tilde{\ma}_{x_2}(x_1,x_3)$ as the hyperbolic angle between the straight lines $\bar{x}_1\bar{x}_2$ and $\bar{x}_2\bar{x}_3$ at $\bar{x}_2$. It can be calculated with the law of cosines using $a_{ij}=\max(\tau(x_i,x_j),\tau(x_j,x_i))$ and $\sigma$, where we set $\sigma=1$ if $x_2$ is not a time endpoint of the triangle and $\sigma=-1$ if $x_2$ is a time endpoint of the triangle. $\sigma$ is called the \emph{sign} of the comparison angle (even though we always have $\tilde{\ma}_{x_2}(x_1,x_3)>0$). For a reduction of the number of case distinctions we also define the \emph{signed comparison angle} $\tilde{\ma}_{x_2}^\mathrm{S}(x_1,x_3)=\sigma\tilde{\ma}_{x_2}(x_1,x_3)$.
\end{definition}

\begin{definition}[Angles]
Let $X$ be a \LpLS and $\alpha,\beta:[0,\varepsilon)\to X$ be two timelike curves (future or past directed or one of each) with $x:=\alpha(0)=\beta(0)$. Then we define the \emph{upper angle} 
\begin{equation*}
\ma_x(\alpha,\beta)=\limsup_{(s,t)\in D, s,t\to 0}\tilde{\ma}_x(\alpha(s),\beta(t))\,,
\end{equation*}
where $D=\{(s,t):s,t>0,\,\alpha(s),\beta(t)\text{ timelike related}\}$. If the limes superior is a limit and finite, we say \emph{the angle exists} and call it an \emph{angle}.

Note that the sign of the comparison angle is independent of $(s,t)\in D$. We define the \emph{sign} of the (upper) angle $\sigma$ to be that sign, and define the \emph{signed (upper) angle} to be $\ma_x^{\mathrm{S}}(\alpha,\beta)=\sigma \ma_x(\alpha,\beta)$.

If maximisers between any two timelike related points are unique (as e.g.\ in $\R^{1,1}$), then we simply write $\ma_p(x,y)$ for the angle at $p$ between the maximisers from $p$ to $x$ and $p$ to $y$.
\end{definition}

We give an alternative definition of timelike curvature bounds in the case $k=0$:

\begin{definition}\label{def:monotonicityComp}
Let $X$ be a \LpLSn. We say that $X$ has \emph{timelike curvature bounded below by $k=0$ (bounded above by $k=0$) in the sense of monotonicity comparison} if every point in $X$ has a neighbourhood $U$ such that
\begin{itemize}
\item[(i)] $\tau|_{U\times U}$ is finite and continuous.
\item[(ii)] $U$ is \emph{timelike geodesically connected}, i.e.\ whenever $x,y \in U$ with $x\ll y$, there exists a future-directed timelike distance realiser $\alpha$ in $U$ from $x$ to $y$, and any distance realiser from $x$ to $y$ contained in $U$ is timelike.
\item[(iii)] Whenever $\alpha: [0, a] \to U$, $\beta: [0, b] \to U$ are timelike distance realisers in $U$ with $x := \alpha(0) = \beta(0)$, we define the function $\theta: (0, a] × (0, b] \supset D \to \R$ by
$\theta(s,t) :=\tilde{\ma}_x^{\mathrm{S}}(\alpha(s), \beta(t))$,
where $(s, t) \in D$ precisely when $\alpha(s), \beta(t)$ are timelike related. We require this to be monotonically increasing (decreasing) in $s$ and $t$.

Additionally, in the case of non-positive timelike curvature, we consider $0 < s' \leq s$ and $0 < t' \leq t$ and require that if $\theta(s', t')$ is not defined (i.e.\ $\alpha(s'), \beta(t')$  are not timelike related) but $\theta(s, t)$ is, then also the comparison points for $\alpha(s'), \beta(t')$ in a comparison triangle for $\Delta(x,\alpha(s),\beta(t))$ are not timelike related.
\end{itemize}
We call the curvature bound \textit{global} if $X$ is a comparison neighbourhood in the above sense.
\end{definition}

\begin{Theorem}[Timelike curvature: Equivalence of definitions]
\label{Theorem: equivalencecurvandmonotonicity}
Let $X$ be a locally timelike geodesically connected \LpLSn. Then $X$ has non-negative (non-positive) timelike curvature in the sense of Definition \ref{Definition: local tl curv bound} if and only if it has non-negative (non-positive) timelike curvature in the sense of monotonicity comparison (see Definition \ref{def:monotonicityComp}).
\end{Theorem}
\begin{proof}
See \cite[Theorem 4.12]{beran2022angles}.
\end{proof}

\begin{remark}
Note that if $X$ is \textit{globally timelike geodesically connected}, i.e.\ between any two timelike related points there is a timelike distance realiser connecting them, and any distance realiser connecting them is timelike (which is certainly true if $X$ is a regularly localisable and globally hyperbolic Lorentzian length space), then $X$ has global non-negative (non-positive) timelike curvature in the sense of Definition \ref{definition: globalcurvbounds} if and only if it has global non-negative (non-positive) timelike curvature in the sense of monotonicity comparison.
\end{remark}

The timelike geodesic prolongation property plays a technical role in our proof of the splitting theorems mainly because of the following result.

\begin{Proposition}[Continuity of angles in spaces with timelike curvature bounded below]
\label{proposition: continuityofangles}
Let $X$ be a strongly causal, localisable, timelike geodesically connected and locally causally closed \LpLS with global non-negative timelike curvature and which satisfies timelike geodesic prolongation. Let $\alpha_n,\alpha,\beta_n,\beta$ be future or past directed timelike geodesics all starting at $\alpha_n(0)=\alpha(0)=\beta_n(0)=\beta(0)=:x$ and with $\alpha_n\to \alpha$ and $\beta_n\to \beta$ pointwise (in particular, $\alpha_n$ is future directed if and only if $\alpha$ is, and similarly for $\beta_n$ and $\beta$). Then
\begin{equation*}
\ma_x(\alpha,\beta)=\lim_n\ma_x(\alpha_n,\beta_n)
\end{equation*}
\end{Proposition}
\begin{proof}
See \cite[Proposition 4.14]{beran2022angles}.
\end{proof}

The following two propositions will be proven together.

\begin{Proposition}[Alexandrov lemma: across version]
\label{lem: alexlem across}
Let $X$ be a \LpLSn.
Let $\Delta:=\Delta(x, y, z)$ be a timelike triangle (in particular the distance realisers between the endpoints exist). Let $p$ be a point on the side $xz$ with $p\ll y$, such that the distance realiser between $p$ and $y$ exists. 
Then we can consider the smaller triangles $\Delta_1:=\Delta(x,p,y)$ and $\Delta_2:=\Delta(p,y,z)$. We construct a comparison situation consisting of a comparison triangle $\bar{\Delta}_1$ for $\Delta_1$ and $\bar{\Delta}_2$ for $\Delta_2$, with $\bar{x}$ and $\bar{z}$ on different sides of the line through $\bar{p}\bar{y}$ and a comparison triangle $\tilde{\Delta}$ for $\Delta$ with a comparison point $\tilde{p}$ for $p$ on the side $xz$. This contains the subtriangles $\tilde{\Delta}_1:=\Delta(\tilde{x},\tilde{y},\tilde{p})$ and $\tilde{\Delta}_2:=\Delta(\tilde{p},\tilde{y},\tilde{z})$, see Figure \ref{fig: alexlem across concave}.

\begin{figure}
\begin{center}
\begin{tikzpicture}
\draw (-0.5693860319981044,1.9834819014638239)-- (0,0);
\draw (2.8896135929029714,2.2006721639230697)-- (4,0);
\draw (0,0)-- (-0.015287989182225509,1.3000898902049949);
\draw (-0.015287989182225509,1.3000898902049949)-- (-0.5693860319981044,1.9834819014638239);
\draw (-0.5693860319981044,1.9834819014638239)-- (1.0360567397613818,4.954583800321347);
\draw (1.0360567397613818,4.954583800321347)-- (-0.015287989182225509,1.3000898902049949);
\draw (2.8896135929029714,2.2006721639230697)-- (3.647225049614162,4.812946100427443);
\draw (3.647225049614162,4.812946100427443)-- (4,0);
\draw [dashed] (2.8896135929029714,2.2006721639230697)-- (3.904456784270502,1.303506235532433);
\begin{scriptsize}
\coordinate [circle, fill=black, inner sep=0.7pt, label=270: {$\bar{x}$}] (A1) at (0,0);
\coordinate [circle, fill=black, inner sep=0.7pt, label=0: {$\bar{p}$}] (A1) at (-0.015287989182225509,1.3000898902049949);
\coordinate [circle, fill=black, inner sep=0.7pt, label=180: {$\bar{y}$}] (A1) at (-0.5693860319981044,1.9834819014638239);
\coordinate [circle, fill=black, inner sep=0.7pt, label=90: {$\bar{z}$}] (A1) at (1.0360567397613818,4.954583800321347);
\coordinate [circle, fill=black, inner sep=0.7pt, label=270: {$\tilde{x}$}] (A1) at (4,0);
\coordinate [circle, fill=black, inner sep=0.7pt, label=180: {$\tilde{y}$}] (A1) at (2.8896135929029714,2.2006721639230697);
\coordinate [circle, fill=black, inner sep=0.7pt, label=90: {$\tilde{z}$}] (A1) at (3.647225049614162,4.812946100427443);
\coordinate [circle, fill=black, inner sep=0.7pt, label=0: {$\tilde{p}$}] (A1) at (3.904456784270502,1.303506235532433);
\end{scriptsize}
\end{tikzpicture}
\end{center}
\caption{A concave situation in the across version.}
\label{fig: alexlem across concave}
\end{figure}
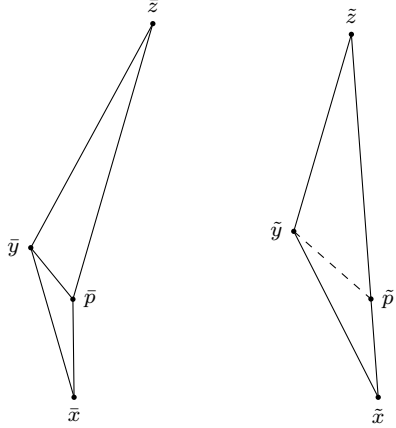

Then the situation $\bar{\Delta}_1$, $\bar{\Delta}_2$ is convex (concave) at $p$ (i.e.\ $\tilde{\ma}_p(x,y)=\ma_{\bar{p}}(\bar{x},\bar{y})\geq\ma_{\bar{p}}(\bar{y},\bar{z})=\tilde{\ma}_p(y,z)$ (or $\leq$)) if and only if $\tau(p,y)\leq\bar{\tau}(\bar{p},\bar{y})$ (or $\geq$).
If this is the case, we have that 
\begin{itemize}
\item each angle in the triangle $\bar{\Delta}_1$ is $\geq$ (or $\leq$) than the corresponding angle in the triangle $\tilde{\Delta}_1$,
\item each angle in the triangle $\bar{\Delta}_2$ is $\geq$ (or $\leq$) than the corresponding angle in the triangle $\tilde{\Delta}_2$.
\end{itemize}
In any case, we have that
\begin{itemize}
\item $\ma_{\bar{y}}(\bar{x},\bar{z}) \geq \ma_{\tilde{x}}(\tilde{x},\tilde{z})=\tilde{\ma}_y(x,z)$.
\end{itemize}
The same is true if $p$ is a point on the side $xz$ such that $y \ll p$.
Note that if $X$ has non-negative (non-positive) timelike curvature and $\Delta$ is within a comparison neighbourhood, the condition is satisfied, i.e.\ $\tau(p,y)\leq\bar{\tau}(\tilde{p},\tilde{y})$ (or $\geq$).
\end{Proposition}

\begin{Proposition}[Alexandrov lemma: future version]
\label{lem: alexlem future}
Let $X$ be a \LpLSn. 
Let $\Delta:=\Delta(x,y,z)$ be a timelike triangle (in particular the distance realisers between the endpoints exist). Let $p$ be a point on the side $xy$, such that the distance realiser between $p$ and $z$ exists. 
Then we can consider the smaller triangles $\Delta_1:=\Delta(x,p,z)$ and $\Delta_2:=\Delta(p,y,z)$. We construct a comparison situation consisting of a comparison triangle $\bar{\Delta}_1$ for $\Delta_1$ and $\bar{\Delta}_2$ for $\Delta_2$, with $\bar{x}$ and $\bar{y}$ on different sides of the line through $\bar{p}\bar{z}$ and a comparison triangle $\tilde{\Delta}$ for $\Delta$ with a comparison point $\tilde{p}$ for $p$ on the side $xy$. This contains the subtriangles $\tilde{\Delta}_1:=\Delta(\tilde{x},\tilde{p},\tilde{z})$ and $\tilde{\Delta}_2:=\Delta(\tilde{p},\tilde{y},\tilde{z})$, see Figure \ref{fig: alexlem future convex}.

\begin{figure}
\begin{center}
\begin{tikzpicture}
\draw (7,0)-- (6.236686420076234,1.0356870286413933);
\draw (6.236686420076234,1.0356870286413933)-- (7.548286699820673,4.0374024205174575);
\draw (7.548286699820673,4.0374024205174575)-- (7,0);
\draw (6.15561793604921,1.839784099337744)-- (6.236686420076234,1.0356870286413933);
\draw (6.15561793604921,1.839784099337744)-- (7.548286699820673,4.0374024205174575);
\draw (9.204717676576246,1.6977850199451887)-- (10,0);
\draw (9.204717676576246,1.6977850199451887)-- (10.90284745202584,4.1006259914346685);
\draw (10.90284745202584,4.1006259914346685)-- (10,0);
\draw [dashed] (9.628868249068917,0.7922996759744166)-- (10.90284745202584,4.1006259914346685);
\begin{scriptsize}
\coordinate [circle, fill=black, inner sep=0.7pt, label=270: {$\tilde{x}$}] (A1) at (10,0);
\coordinate [circle, fill=black, inner sep=0.7pt, label=180: {$\tilde{p}$}] (A1) at (9.628868249068917,0.7922996759744166);
\coordinate [circle, fill=black, inner sep=0.7pt, label=270: {$\bar{x}$}] (A1) at (7,0);
\coordinate [circle, fill=black, inner sep=0.7pt, label=180: {$\bar{p}$}] (A1) at (6.236686420076234,1.0356870286413933);
\coordinate [circle, fill=black, inner sep=0.7pt, label=90: {$\bar{z}$}] (A1) at (7.548286699820673,4.0374024205174575);
\coordinate [circle, fill=black, inner sep=0.7pt, label=180: {$\bar{y}$}] (A1) at (6.15561793604921,1.839784099337744);
\coordinate [circle, fill=black, inner sep=0.7pt, label=180: {$\tilde{y}$}] (A1) at (9.204717676576246,1.6977850199451887);
\coordinate [circle, fill=black, inner sep=0.7pt, label=90: {$\tilde{z}$}] (A1) at (10.90284745202584,4.1006259914346685);
\end{scriptsize}
\end{tikzpicture}
\end{center}
\caption{A convex situation in the future version.}
\label{fig: alexlem future convex}
\end{figure}
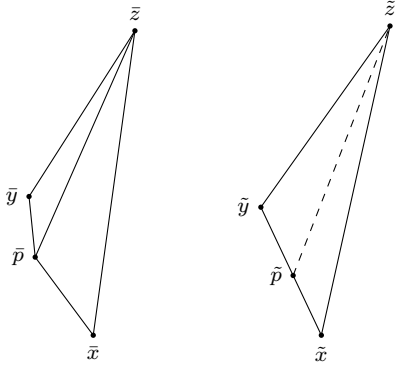

Then the situation $\bar{\Delta}_1$,$\bar{\Delta}_2$ is convex (concave) at $p$ (i.e.\
$\ma_{\bar{p}}(\bar{y},\bar{z})\geq\ma_{\bar{p}}(\bar{x},\bar{z})$ (or $\leq$)) if and only if $\tau(p,z)\leq\bar{\tau}(\bar{p},\bar{z})$ (or $\geq$). If this is the case, we have that
\begin{itemize}
\item each angle in the triangle $\bar{\Delta}_1$ is $\geq$ (or $\leq$) than the corresponding angle in the triangle $\tilde{\Delta}_1$,
\item each angle in the triangle $\bar{\Delta}_2$ is $\leq$ (or $\geq$) than the corresponding angle in the triangle, $\tilde{\Delta}_2$.
\end{itemize}
In any case, we have that
\begin{itemize}
\item $\ma_{\bar{z}}(\bar{x},\bar{y})\leq\ma_{\tilde{z}}(\tilde{x},\tilde{y})=\tilde{\ma}_z(x,y)$.
\end{itemize}
Note that if $X$ has non-negative (non-positive) timelike curvature and $\Delta$ is within a comparison neighbourhood, the condition is satisfied, i.e.\ $\tau(p,z)\leq\bar{\tau}(\tilde{p},\tilde{z})$ (or $\geq$).
\end{Proposition}
\begin{proof}[Proof for both Alexandrov lemmas]
It is sufficient to show the ``if'' part of the statement. Indeed, under the assumption that the triangles $\Delta_1$ and $\Delta_2$ are non-degenerate, the ``if'' part holds with strict inequalities, which then also implies the ``only if'' part of the statement. The degenerate cases are trivial.

For the triangles $\bar{\Delta}_1$ vs.\ $\tilde{\Delta}_1$, only one side-length changes, and it is not the longest side of the triangle in both versions. Thus, the monotonicity statement in the law of cosines, cf.\ Lemma \ref{Lemma: hyperboliclawofcosines}, implies that in the across version,
the relation between the desired
angle $\ma_{\bar{p}}(\bar{x},\bar{y})$ in $\bar{\Delta}_1$ to the corresponding angle $\ma_{\tilde{p}}(\tilde{x},\tilde{y})$ in $\tilde{\Delta}_1$ is pointing in the opposite
direction than the relation of the changing side-length $\tau(p,y)=\bar{\tau}(\bp,\by)$ to the side-length $\bar{\tau}(\tilde{p},\tilde{y})$. 
In the future version, we obtain that the relation between the 
desired angle $\ma_{\bar{p}}(\bar{x},\bar{z})$ in $\bar{\Delta}_1$ to the corresponding angle $\ma_{\tilde{p}}(\tilde{x},\tilde{z})$ in $\tilde{\Delta}_1$ is pointing in the other direction than the relation of the changing side-length $\tau(p,z)=\bar{\tau}(\bp,\bz)$ to the side-length $\bar{\tau}(\tilde{p},\tilde{z})$.

Similarly, for the triangles $\bar{\Delta}_2$ vs.\ $\tilde{\Delta}_2$, only one side-length changes. In the across version, it is not the longest side, and for the future version, it is. 
Thus, in the future version, the monotonicity statement in the law of cosines yields that the relation between the angle $\ma_{\bar{p}}(\bar{y},\bar{z})$ in $\bar{\Delta}_1$ to the corresponding angle $\ma_{\tilde{p}}(\tilde{y},\tilde{z})$ in $\tilde{\Delta}_1$ is pointing in the 
same direction as the relation of the changing side-length $\bar{\tau}(\bp,\by)$ to the side-length $\bar{\tau}(\tilde{p},\tilde{y})$.
Similarly, in the across version, we get that the relations between the angles $\ma_{\bar{p}}(\bar{y},\bar{z})$ in $\bar{\Delta}_1$ and $\ma_{\tilde{p}}(\tilde{y},\tilde{z})$ in $\tilde{\Delta}_1$ points in the opposite direction as the relation of the changing side-length $\bar{\tau}(\bp,\by)$ to the side-length $\bar{\tau}(\tilde{p},\tilde{y})$.
Note that $\tilde{\Delta}_1$ and $\tilde{\Delta}_2$ together form the triangle $\tilde{\Delta}$, so we have $\ma_{\tilde{p}}(\tilde{x},\tilde{y})=\ma_{\tilde{p}}(\tilde{y},\tilde{z})$ in the across version and $\ma_{\tilde{p}}(\tilde{x},\tilde{z})=\ma_{\tilde{p}}(\tilde{y},\tilde{z})$ in the future version. From this, the desired inequalities follow.

For the "split" angle, we use the triangle equality along lines and the reverse triangle inequality on the broken side, giving $\bar{\tau}(\bar{x},\bar{z})\geq\bar{\tau}(\bar{x},\bar{p})+\bar{\tau}(\bar{p},\bar{z})=\tau(x,p)+\tau(p,z)= \tau(x,z)=\bar{\tau}(\tilde{x},\tilde{z})$  in the across statement and $\bar{\tau}(\bar{x},\bar{y})\geq\bar{\tau}(\bar{x},\bar{p})+\bar{\tau}(\bar{p},\bar{y})=\tau(x,p)+\tau(p,y)= \tau(x,y)=\bar{\tau}(\tilde{x},\tilde{y})$ in the future statement. For the triangles $\bar{\Delta}=\Delta(\bar{x},\bar{y},\bar{z})$ and $\tilde{\Delta}=\Delta(\tilde{x}, \tilde{y}, \tilde{z})$, only one side-length changes. In the across version it is the longest side of the triangle, and in the future version, it is not. 
Thus, in the future version, the monotonicity statement in the law of cosines implies that the relation between the angle $\ma_{\bar{z}}(\bar{x},\bar{y})$ in $\bar{\Delta}$ to the corresponding angle $\ma_{\tilde{z}}(\tilde{x},\tilde{y})$ in $\tilde{\Delta}$ is pointing in the opposite direction than the relation of the changing side-lengths $\tau(x,y)=\bar{\tau}(\bx,\by)$ and $\bar{\tau}(\tilde{x},\tilde{y})$. 
Similarly, in the across version, we obtain that the relation of the angle $\ma_{\bar{y}}(\bar{x},\bar{z})$ in $\bar{\Delta}$ and $\ma_{\tilde{y}}(\tilde{x},\tilde{z})$ in $\tilde{\Delta}$ points in the same direction as the inequality between $\tau(x,z)=\bar{\tau}(\bx,\bz)$ and $\bar{\tau}(\tilde{x},\tilde{z})$.

Finally, note that all monotonicity arguments work the same if the timelike relation between $y$ and $p$ is reversed in the across version.
\end{proof}

\subsection{Extensions of Lorentzian length spaces}\label{subsec:LLS:extensions}

As in \cite{beem1985toponogov}, in the proof of the splitting theorem we will first show that $I(\gamma):=I^+(\gamma)\cap I^-(\gamma)$ splits, and then we will use an inextendibility argument to show that $I(\gamma)=X$. To this end, let us recall the notion of extensions of Lorentzian (pre-)length spaces as introduced in \cite[Def.\ 3.1]{grant2019inextendibility}.

\begin{definition}[Extensions]\par 
\label{definition: extensions}
Let $(X,d,\ll,\leq,\tau)$ be a Lorentzian pre-length space. A Lorentzian pre-length space $(\tilde{X},\tilde{d},\tilde{\ll},\tilde{\leq},\tilde{\tau})$ is called an \emph{extension} of $(X,d,\ll,\leq,\tau)$ if
\begin{enumerate}
    \item $(\tilde{X},\tilde{d})$ is connected,
    \item there exists an isometric embedding $\iota:(X,d) \to (\tilde{X},\tilde{d})$,
    \item $\iota(X) \subsetneqq\tilde{X}$ is open,
    \item $\iota$ preserves $\leq$ and $\ll$,
    \item $\iota$ preserves $\tau$-lengths of causal curves.
\end{enumerate}
\end{definition}

\begin{remark}[Convex extensions]\label{remark: convex_extensions}
An important class of examples (and the only one we will need) is the case of open, causally convex subsets of a suitable Lorentzian length space. In this case, $\iota$ even preserves $\tau$.
\end{remark}

Next, we recall the definition of \textit{timelike completeness} for (localisable) Lorentzian pre-length spaces, cf.\ \cite[Def.\ 5.1]{grant2019inextendibility}. However, since continuous extendibility, extendibility as a causal curve and geodesic extendibility are in general inequivalent concepts, we will define the timelike completeness property precisely in such a way that it is sufficient for the subsequent inextendibility result.

\begin{definition}[Timelike completeness (TC)]\par 
Let $(X,d,\ll,\leq,\tau)$ be a localisable Lorentzian pre-length space. We say $X$ satisfies the \emph{timelike completeness (TC) property} if each (future or past directed) timelike geodesic $\alpha:[a,b) \to X$ that is inextendible to $[a,b]$ as a continuous curve has infinite $\tau$-length.
\end{definition}

\begin{Theorem}[Inextendibility of (TC) spaces]\par 
\label{Theorem: TCinextendibility}
Let $(X,d,\ll,\leq,\tau)$ be a strongly causal Lorentzian length space satisfying the (TC) property. Then $X$ is inextendible as a regularly localisable Lorentzian length space.
\begin{proof}
See \cite[Thm.\ 5.3]{grant2019inextendibility}.
\end{proof}
\end{Theorem}

\section{Lorentzian products}\label{sec:products}

This section is dedicated to the treatment of Lorentzian product spaces. In the context of Lorentzian pre-length spaces, (warped) products were first introduced in \cite{alexander2019generalized}. On the one hand, this investigation is too general for our purposes, we only need products instead of warped products.
On the other hand, the formulation we use also works for non-intrinsic spaces.

\begin{definition}[Lorentzian product]
\label{def: lor product}
Let $(X,d)$ be a metric space. Define the space $Y:= \R \times X$. Equip it with the product metric 
\begin{equation*}
D: Y \times Y \to \R, \ D((s,x),(t,y)):=\sqrt{|t-s|^2+d(x,y)^2}.    
\end{equation*}
Define the timelike relation as $(s,x) \ll (t,y) :\iff t-s > d(x,y)$ and the causal relation as $(s,x) \leq (t,y) :\iff t-s \geq d(x,y)$. 
Define the product time separation $\tau: Y \times Y \to \R$ via 
\begin{equation*}
\tau((s,x),(t,y)):=\sqrt{(t-s)^2 - d(x,y)^2}    
\end{equation*}
if $(s,x) \leq (t,y)$, and 0 otherwise. 

Then $(Y,D,\ll,\leq,\tau)$ is called the Lorentzian product of $X$ with $\R$. If it is clear that the Lorentzian product is meant, we simply denote it by $\R\times X$.
\end{definition}

The following is more or less immediate from the definition:

\begin{Proposition}[Properties of Lorentzian products]
\label{Proposition: Lorproductscontintaucausclosed}
Let $(X,d)$ be a metric space. Then the Lorentzian product $(Y,D,\ll,\leq,\tau)$ is a Lorentzian pre-length space. Moreover, $\tau$ is even continuous and $\leq$ is closed on $Y \times Y$.
\end{Proposition}

\begin{proof}
$\tau$ is continuous by definition, so in particular, lower semi-continuous. 
The inclusion of $\ll$ in $\leq$ is clear from the definition as well. The transitivity of both relations as well as the reverse triangle inequality of $\tau$ for $\leq$-related triples follow (via an elementary calculation) from the triangle inequality for $d$. 
The reflexivity of $\leq$ is clear from the definition.
Finally, the definition of $\ll$ is a reformulation of $\tau >0$. 

Note that for converging real sequences $n_i \to n, m_i \to m$ such that $n_i \leq m_i$ for all $i$, we have $n \leq m$. The closedness of $\leq$ on $Y \times Y$ follows immediately from this fact.
\end{proof}

It is easily seen that generalised cones in the sense of \cite{alexander2019generalized} with warping function $f \equiv 1$ are Lorentzian products in the sense of Definition \ref{def: lor product}. 
Indeed, in this case we have, in the notation of \cite[Def.\ 3.9]{alexander2019generalized}, that $m_{s,t}:= \min_{r \in [s,t]}f(r)=1$ for all $s,t \in \R$. Then \cite[Lem.\ 3.10]{alexander2019generalized} shows that a causal curve in the sense of \cite[Def.\ 3.2]{alexander2019generalized} is a causal curve in the usual sense in our setting, with timelikeness being inherited as well. Moreover, the variational length of \cite[Def.\ 3.9]{alexander2019generalized} is precisely the $\tau$-length of a causal curve in the sense of Definition \ref{definition: taulength}. It is easily seen that if $X$ is strictly intrinsic, then the product $\R \times X$ in our sense can be canonically identified with the warped product $\R \times_{f} X$ with $f \equiv 1$.

We now show that Lorentzian products are always strongly causal and non-totally imprisoning.

\begin{Proposition}[Diamonds form basis]
\label{prop: diamondbasis}
Let $Y:=\R \times X$ be a Lorentzian product.\ Then $\{I(p,q) \mid p,q \in Y\}$ forms a basis for the topology. In particular, $Y$ is strongly causal.
\end{Proposition}

\begin{proof}
Consider the open set $O:=(a,c) \times B_R(x)$ for $a,c \in \R, R >0, x \in X$ and let $(b,y)\in O$. Choose $\varepsilon=\min(b-a,c-b,R-d(x,y))>0$, then $p:=(b-\varepsilon,y),q:=(b+\varepsilon,y) \in \bar{O}$. Then we have that $I(p,q)\subseteq O$: If $(s,z) \in I(p,q)$, then $s \in (b-\varepsilon,b+\varepsilon) \subseteq (a,c)$. In particular, at least one of $|s-(b-\varepsilon)|\leq\varepsilon$ and $|s-(b+\varepsilon)|\leq\varepsilon$ holds. By definition of $\ll$, we then also have $d(x,z)\leq d(x,y)+d(y,z)<R$, so $y \in B_R(x)$. Thus, $(s,z) \in O$.
\end{proof}

\begin{Proposition}[Non-total imprisonment of Lorentzian products]
\label{proposition: nontotalimprisonmentproducts}
Any Lorentzian product $Y = \R \times X$, where $(X,d)$ is a metric space, is non-totally imprisoning.
\begin{proof}
For any two points $(s,x)\leq(t,y)\in Y$ we have $d(x,y)\leq t-s$, so  $D((s,x),(t,y))=\sqrt{(t-s)^2+d(x,y)^2}\leq\sqrt{2}(t-s)$. In particular, for any future directed causal curve $\gamma:[a,b]\to Y$ starting at $\gamma(a)=(s,x)$ and ending at $\gamma(b)=(t,y)$ we have $L_D(\gamma)\leq\sqrt{2}(t-s)$: In any partition of $[a,b]$, the above bound in any of the subintervals forms a telescopic sum, making $\sqrt{2}(t-s)$ an upper bound on the length of the polygon approximation of $\gamma$ corresponding to the partition. Let now $K$ be compact, then we can enclose it in a bounded set: $K\subseteq [s,t]\times \bar{B}_R(x)$ for some $s<t$, $R>0$ and $x\in X$. We know that the $D$-length of future directed causal curves contained in $[s,t]\times \bar{B}_R(x)$ is bounded by $\sqrt{2}(t-s)$, so we have that $Y$ is non-totally imprisoning.
\end{proof}
\end{Proposition}

The next result characterises distance realisers in $Y = \R \times X$. Note that we make use of the following: A distance minimiser in a metric space can always be parametrised by unit speed, which is evidently a Lipschitz parametrisation, cf.\ \cite[Proposition 2.5.9]{burago2001course}. 

\begin{Proposition}[Characterisation of distance realisers]
\label{Proposition: distancerealizersinproducts}
Let $(X,d)$ be a metric space and let $Y:=\R \times X$ the Lorentzian product. Then a continuous causal curve $\gamma=(\alpha,\beta): [a,b] \to Y$ is a distance realiser if and only if either $\beta = const$ or $\beta:[a,b] \to X$ is a (metric) distance realiser and when reparametrising $\gamma$ such that $\beta$ is unit speed parametrised, $\alpha:[a,b]\to\R$ is affine, i.e. of the form $\alpha(t)=ct+d$ with $c=1$ (which implies that $\gamma$ future directed null) or $c>1$ (which implies that $\gamma$ is future directed timelike). In particular, $\gamma$ has causal character and a reparametrisation as a (Lipschitz) causal curve. If in addition $Y$ is localisable, then the same is true for continuous geodesics.
\end{Proposition}

\begin{proof}
$\gamma$ is a distance realiser if and only if for all $r<s<t$ we have 
\begin{equation*}
\tau(\gamma(r),\gamma(t))=\tau(\gamma(r),\gamma(s))+\tau(\gamma(s),\gamma(t)).\quad (\star)
\end{equation*}
We denote $\gamma(r)=(t_1,x_1)$, $\gamma(s)=(t_2,x_2)$ and $\gamma(t)=(t_3,x_3)$, then $(\star)$ reads:
\begin{equation*}
\sqrt{(t_3-t_1)^2-d(x_1,x_3)^2}=\sqrt{(t_2-t_1)^2-d(x_1,x_2)^2}+\sqrt{(t_3-t_2)^2-d(x_2,x_3)^2}
\end{equation*}
We now define $\lambda=\frac{t_2-t_1}{t_3-t_1}$ and the function $f(x)=\sqrt{1-x^2}$. Then $(\star)$ simplifies to
\begin{equation*}
f\left(\frac{d(x_1,x_3)}{t_3-t_1}\right)=\lambda f\left(\frac{d(x_1,x_2)}{\lambda(t_3-t_1)}\right)+(1-\lambda)f\left(\frac{d(x_2,x_3)}{(1-\lambda)(t_3-t_1)}\right).\quad (\star\star)
\end{equation*}
We use that $f$ is concave and monotonously decreasing on $[0,1]$ and $d(x_1,x_2)+d(x_2,x_3)\geq d(x_1,x_3)$ on the right hand side of $(\star\star)$:
\begin{align*}
&\lambda f\left(\frac{d(x_1,x_2)}{\lambda(t_3-t_1)}\right)+(1-\lambda)f\left(\frac{d(x_2,x_3)}{(1-\lambda)(t_3-t_1)}\right)\\
&\leq f\left(\frac{d(x_1,x_2) + d(x_2,x_3)}{(t_3-t_1)}\right)\leq f\left(\frac{d(x_1,x_3)}{(t_3-t_1)}\right)\,.
\end{align*}
So we get that $(\star\star)$ is equivalent to having equality here. As $f$ is strictly concave, the first inequality has equality if and only if
\begin{equation}\label{popDistRealizInProd:eqAffine}
\frac{d(x_1,x_2)}{t_2-t_1}=\frac{d(x_2,x_3)}{t_3-t_2}\,,
\end{equation}
where we got rid of $\lambda$ again for generality. As $f$ is strictly monotonously decreasing, the second inequality has equality if and only if
\begin{equation}\label{popDistRealizInProd:eqSpaceDistRealiz}
d(x_1,x_2)+d(x_2,x_3)=d(x_1,x_3)\,.
\end{equation}

We translate this back to what happens for general $r<s<t$: Equation \ref{popDistRealizInProd:eqSpaceDistRealiz} reads $d(\beta(r),\beta(s))+d(\beta(s),\beta(t))= d(\beta(r),\beta(t))$, i.e., it holds if and only if $\beta$ is a distance realiser. Equation \ref{popDistRealizInProd:eqAffine} reads $\frac{d(\beta(r),\beta(s))}{\alpha(s)-\alpha(r)}=\frac{d(\beta(s),\beta(t))}{\alpha(t)-\alpha(s)}$, i.e.\ that quantity is a constant $\frac{1}{c}$ (if $\frac{1}{c}$ is $0$ we get that $\beta$ is constant, in this case the result is trivial). In particular, if we reparametrise $\gamma$ such that $\beta$ is unit speed (i.e.\ $d(\beta(s),\beta(t))=t-s$), $\alpha(t)-\alpha(s)=c(t-s)$, i.e.\ $\alpha$ is affine. 

Now we look at what happens for different parameters $c$: If $c<1$, we easily see that $f(c)>0$, so $\tau(\gamma(r),\gamma(t))=(\alpha(t)-\alpha(r))f(c)>0$ and the curve is timelike. For $c=1$, $f(c)=f(1)=0$, so $\alpha(t)-\alpha(r)=d(\beta(r),\beta(t))$ and the curve is null.

Now if $Y$ is localisable and $\gamma$ is only a geodesic, we can cover the domain with intervals $J_i$ that are open in $[a,b]$ where it is distance realising. Then on each of these $J_i$ we apply the above. As the $J_i$ overlap, the constant $c$ agrees, and we get that $\beta|_{J_i}$ is distance realising, thus $\beta$ is a geodesic. The converse follows equivalently. We even get that the subintervals of the domain where the curves are distance realising agree. 
\end{proof}

\begin{Corollary}[$\tau$-arclength parametrisations of distance realisers]
\label{Corollary: tauarclengthofmaxinproducts}
Any continuous timelike maximiser $\gamma:[a,b] \to \R \times X$ has a Lipschitz $\tau$-arclength parametrisation, which is of the form $(\alpha,\beta)$, with $\alpha:[0,L_\tau(\gamma)]\to\R$, $\alpha(t)=ct+d$, and $\beta:[0,L_\tau(\gamma)]\to X$ a constant speed minimiser of speed $\sqrt{c^2-1}$. If $\R \times X$ is localisable, then the same is true for timelike geodesics (with $\beta$ constant speed geodesic).
\end{Corollary}

Note that the results above continue to hold for maximisers resp.\ geodesics defined on any interval $I$, not just a closed one.

To finish this section, we show that global hyperbolicity of the Lorentzian product $Y=\R \times X$ is equivalent to the metric properness of $X$.

\begin{Proposition}[Global hyperbolicity and properness]
\label{Proposition: productglobhypiffmetricproper}
A Lorentzian product $Y:=\R \times X$ is globally hyperbolic if and only if $X$ is a proper metric space.
\end{Proposition}
\begin{proof}
Since products are always non-totally imprisoning by Proposition \ref{proposition: nontotalimprisonmentproducts}, we only need to check that causal diamonds in $Y$ are compact if and only if $(X,d)$ is a proper metric space. First, suppose that $Y$ is globally hyperbolic. Let $(t,x) \in Y$ and fix $R>0$. Consider the points $p:=(t-R,x), q:=(t-R,x)$. Then $(t,x) \in J(p,q)$. 
By definition of $\leq$, any point $(t,y)$ with the same $\R$-coordinate as $(t,x)$ is in $J(p,q)$ if and only if $d(x,y)\leq R$. 
Thus, the set $J(p,q) \cap \{(t,x) \mid x \in X\}$, which is compact by assumption, can be identified with the closed ball $\bar{B}_R(x)$ in $X$ of radius $R$ around $x$, establishing the fact that closed balls in $X$ are compact.

Conversely, suppose that $X$ is proper and let $p=(r,x), q=(t,z) \in Y$ with $p \leq q$. By definition of $\leq$, if $(s,y) \in J(p,q)$, then $r \leq s \leq t$. 
Moreover, we have $|s| \leq |r| + |t|$. 
Set $R:=2|r|+2|t|$. 
Then for $(s,y) \in J(p,q)$, we also have $d(x,y) \leq |s-r| \leq |s|+|r| \leq R$. Thus, $y \in \bar{B}_R(x)$, which is compact by assumption. In total, we conclude that $J(p,q) \subseteq [r,t] \times \bar{B}_R(x)$, which is a compact set as the Cartesian product of two compact sets. The fact that $J(p,q)$ is closed follows immediately from the closedness of $\leq$, so $J(p,q)$ is compact as well. 
\end{proof}

\section{Rays, lines, co-rays and asymptotes}\label{sec:lines}

\subsection{Rays, lines, co-rays, asymptotes, timelike co-ray condition}\label{subsec:lines:lines}

In this subsection, we study causal rays and lines and show how to obtain them as limits of causal maximisers. Moreover, we analyse triangles where one side is a segment on a timelike line and show that the angles adjacent to the line are equal to their comparison angles. This in fact follows from the more general principle that one can stack triangle comparisons of nested triangles with two endpoints on a timelike line. The latter situation arises in the construction of asymptotes, and via the stacking principle, one can show any future directed and any past directed asymptote from a common point to a given timelike line fit together to give a (timelike) asymptotic line. In constructing co-rays and asymptotes we follow \cite{beem1985toponogov}, whereas the stacking principle and equality of angles can be viewed as Lorentzian analogues of \cite[Lem.\ 10.5.4]{burago2001course}.

\begin{definition}[Rays, Lines]\par
Let $(X,d,\ll,\leq,\tau)$ be a Lorentzian pre-length space. A \emph{future directed causal ray} is a future inextendible, future directed causal curve $c:I \to X$ that maximises the time separation between any of its points, where $I$ is either a closed interval $[a,b]$ or a half-open interval $[a,b)$. A \emph{future directed causal line} is a (doubly) inextendible, future directed causal curve $\gamma:I \to X$ that maximises the time separation between any of its points. Here $I$ can in general be open, closed, or half-open. More generally, for $S \subset X$, a \emph{future directed causal $S$-ray} is a future directed, future inextendible causal curve starting in $S$ and satisfying $\tau(S,c(t)):=\sup_{p \in S} \tau(p,c(t)) = L_{\tau}(c|_{[0,t]})$. Past directed rays and lines are defined analogously. A future directed causal ray $c: I \to X$ is called \emph{complete} if $L_{\tau}(c)=\infty$. Similarly, a future directed causal line $\gamma$ is called \emph{complete} if there is $t_0 \in I$ such that the past and future directed rays obtained from $\gamma$ at $\gamma(t_0)$ are complete.
Unless explicitly stated otherwise, all causal rays and lines are understood to be future directed. 
\end{definition}

\begin{remark}
\begin{enumerate}
\item[]
\item Clearly, any ray $c$ is a $\{c(0)\}$-ray. Conversely, if $c:[a,b)\to X$ is a future directed $S$-ray for some $S \subset X$, then $\tau(c(a),c(t)) \leq \tau(S,c(t)) = L_{\tau}(c|_{[a,t]})$, so $c$ is a ray. A similar statement is true for past directed rays.

\item If $X$ is regularly localisable, then any causal ray/line $c$ is either timelike or null by Theorem \ref{theorem: causalcharofmaximizers}.

\item If $X$ is localisable and causally path-connected, then any ray $c$ has to be defined on a half-open interval and any line $\gamma$ has to be defined on an open interval.
\end{enumerate}
\end{remark}

In the following, unless we specify the assumptions on $X$, we always take $X$ to be as in the splitting theorem.

\begin{Proposition}\par 
Let $z_n \to z$ in $X$. Let $p_n \in I^+(z_n)$ and let $c_n:[0,a_n] \to X$ be a sequence of future directed, maximising causal curves from $z_n$ to $p_n$ in $d$-arclength parametrisation. If $\tau(z_n,p_n) \to \infty$, then there is a limit causal ray $c:[0,\infty) \to X$ from $z$.
\begin{proof}
We only have to show $a_n \to \infty$, the rest follows from the limit curve theorem. Suppose $a_n \to a < \infty$. This implies that all $p_n$ are contained in a common compact set $K$: Indeed, if we assume this to be false, then for each $n > 0$ there would be a $p_k$, $k=k(n)$, such that $p_k \notin \overline{B}_n(z)=\{x \in X: d(x,z) \leq n\}$ (recall that $(X,d)$ is proper), which in particular means that
\begin{align*}
    a_{k(n)} = L^d(c_{k(n)}) \geq d(z_{k(n)},p_{k(n)}) \geq d(p_{k(n)},z) - d(z_{k(n)},z) \geq n - d(z_{k(n)},z),
\end{align*}
which is a contradiction since the right hand side tends to $\infty$.
So, up to a choice of subsequence, we may assume that $\gamma_n(a_n) = p_n \to p$. But then, by continuity of $\tau$,
\begin{align*}
    \tau(z,p) = \lim_n \tau(z_n,p_n) = \infty,
\end{align*}
a contradiction to the finiteness of $\tau$ on all of $X \times X$.
\end{proof}
\end{Proposition}

Let now $c:[0,\infty) \to X$ be a complete timelike $S$-ray, and fix $z \in I^-(c) \cap I^+(S)$. Let $z_n \to z$ and $p_n:=c(r_n)$ for some sequence of parameters $r_n \to \infty$. Then it is easily seen that $\tau(z_n,p_n) \to \infty$: Indeed, let $r > 0$ be large and assume w.l.o.g. that all $z_n$ are in $I^-(c(r))$, then
\begin{align*}
    \tau(z_n,p_n) \geq \tau(z_n,c(r)) + \tau(c(r),c(r_n)) \to \infty.
\end{align*}
Hence, by the above result, we may construct causal rays as follows: Let $\mu_n$ be maximising timelike curves from $z_n$ to $p_n$ and let $\mu$ be a limit causal ray from $z$ whose existence we just proved.

\begin{definition}[Co-rays, asymptotes]
\label{definition: corays,asymptotes}
Any causal ray $\mu$ constructed by the above method is called a \emph{co-ray} to the ray $c$ at $z$. If $z_n = z$ for all $n$, then $\mu$ is called an \emph{asymptote} to $c$ at $z$.
\end{definition}

Since maximising causal curves have causal character, a co-ray is either timelike or null.

\begin{definition}[Timelike co-ray condition]
\label{definition: TCRC}
Let $c:[0,\infty) \to X$ be a complete timelike $S$-ray and let $p \in I^-(c) \cap I^+(S)$. We say the \emph{timelike co-ray condition (TCRC)} holds at $p$ if every co-ray to $c$ at $p$ is timelike.
\end{definition}

We turn to the treatment of triangles adjacent to lines and begin with the aforementioned stacking principle. It will be an essential technical tool for controlling the behaviour of asymptotes to a line. The following two results are true for rather general Lorentzian pre-length spaces, the exact conditions are specified.

\begin{Proposition}[Comparison situations stack along a geodesic]
\label{Proposition: stackingprinciple}
Let $X$ be a timelike geodesically connected \LpLS with global timelike curvature bounded below by $0$ and $\gamma:\R\to X$ be a complete timelike line. Let $p\in X$ be a point not on $\gamma$. Let $t_1<t_2<t_3$ such that all $y_i:=\gamma(t_i)$ are timelike related to $p$, see Figure \ref{fig: stacking_domain}. Let $\bar{\Delta}_{12}:=\Delta(\bar{p},\bar{y}_1,\bar{y}_2)$ be a comparison triangle for $\Delta_{12}:=\Delta(p,y_1,y_2)$ and extend the side $\bar{y}_1,\bar{y}_2$ to a comparison triangle $\bar{\Delta}_{23}:=\Delta(\bar{p},\bar{y}_2,\bar{y}_3)$ for $\Delta_{23}:=\Delta(p,y_2,y_3)$. We choose it in such a way that $\bar{y}_1$ and $\bar{y}_3$ lie on opposite sides of the line through $\bar{y}_1,\bar{y}_2$. Then $\bar{y}_1,\bar{y}_2,\bar{y}_3$ are collinear. That makes $\bar{\Delta}_{13}:=\Delta(\bar{p},\bar{y}_1,\bar{y}_3)$ a comparison triangle for $\Delta_{13}:=\Delta(p,y_1,y_3)$.
\end{Proposition}
\begin{proof}
We set $s_-=\sup(\gamma^{-1}(I^-(p)))$ and $s_+=\inf(\gamma^{-1}(I^+(p)))$. Then the set of $s$ where $\gamma(s)$ is timelike related to $p$ is $(-\infty,s_-)\cup(s_+,+\infty)$. We assume $p\ll y_2$, the other case $y_2\ll p$ can be reduced to this one by flipping the time orientation of the space.

%Setup Alexandrov
We create comparison situations for an Alexandrov situation: We take the triangles $\bar{\Delta}_{12}=\Delta(\bar{p},\bar{y}_1,\bar{y}_2)$ and $\bar{\Delta}_{23}=\Delta(\bar{p},\bar{y}_2,\bar{y}_3)$ as in the statement. 
We also create a comparison triangle $\tilde{\Delta}_{13}=\Delta(\tilde{p},\tilde{y}_1,\tilde{y}_3)$ for $(p,y_1,y_3)$ and get a comparison point $\tilde{y}_2$ for $y_2$ on the side $y_1y_3$. Then we can apply curvature comparison to get $\bar{\tau}(\tilde{p},\tilde{y}_2) \geq \tau(p,y_2)=\bar{\tau}(\bar{p},\bar{y}_2)$. By Lemmas \ref{lem: alexlem across} and \ref{lem: alexlem future}\footnote{If $p\ll y_1\ll y_2\ll y_3$, we use Lemma \ref{lem: alexlem future} and if $y_1\ll p\ll y_2\ll y_3$ we use Lemma \ref{lem: alexlem across}.}, this means the situation is convex, i.e.\ 
\begin{equation}\label{pop:compStack:eq:Convex}
\tilde{\ma}_{y_2}(p,y_1)\leq \tilde{\ma}_{y_2}(p,y_3)\,.
\end{equation} 
If $p\ll y_1$ we also get
\begin{equation}\label{pop:compStack:eq:Monotonous}
\tilde{\ma}_{y_1}(p,y_2)\leq\tilde{\ma}_{y_1}(p,y_3)
\end{equation}
and if $y_1\ll p$, this inequality flips.

Now we apply equation (\ref{pop:compStack:eq:Monotonous}) to the situation where we fix $y_2$ and move $y_3$ and $y_1$: We get $\tilde{\ma}_{y_2}(p,\gamma(t_3))$ is monotonously increasing in $t_3$. Similarly, the time-reversed situation gives that if $p\ll \gamma(t_1)$, $\tilde{\ma}_{y_2}(p,\gamma(t_1))$ is monotonously increasing in $t_1$ and if $\gamma(t_1)\ll p$ it is also monotonously increasing in $t_1$ (note it is also increasing when switching from one case to the other). Then remember $\tilde{\ma}_{y_2}(p,\gamma(t_1))\leq \tilde{\ma}_{y_2}(p,\gamma(t_3))$ by the convexity of the Alexandrov situation (equation (\ref{pop:compStack:eq:Convex})). Note that the left hand side is decreasing with decreasing $t_1$ and the right hand side is increasing with increasing $t_3$.
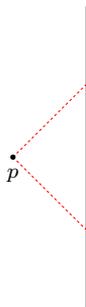
\begin{figure}
\begin{center}
\begin{tikzpicture}
\draw (-1pt,-2) -- (-1pt,2);
\draw[color=red, dash pattern=on 1pt off 1pt] (-1.,0.)-- (0.,1.);
\draw[color=red, dash pattern=on 1pt off 1pt] (-1.,0.)-- (0.,-1.);
\draw[color=green] (0.,1.) -- (0.,2);
\draw[color=green] (0.,-1.) -- (0.,-2);
\begin{scriptsize}
\coordinate [circle, fill=black, inner sep=0.7pt, label=270: {$p$}] (p) at (-1.,0.);
\end{scriptsize}
\end{tikzpicture}
\end{center}
\caption{The domain where the $y_i$ can lie in is in green.}
\label{fig: stacking_domain}
\end{figure}

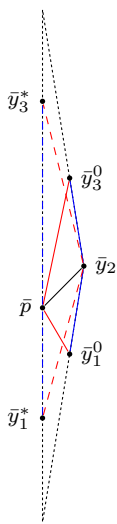
\begin{figure}
\begin{center}
\begin{tikzpicture}[x=0.3cm,y=0.3cm]
\begin{scriptsize}
\coordinate [circle, fill=black, inner sep=0.7pt, label=360: {$\bar{y}_1^0$}] (y10) at (0.7055951711559928,-2.174402745312005);
\coordinate [circle, fill=black, inner sep=0.7pt, label=360: {$\bar{y}_3^0$}] (y30) at (0.7055951711559928,5.606824634988012);
\coordinate [circle, fill=black, inner sep=0.7pt, label=180: {$\bar{p}$}] (p) at (-0.4838237980000031,-0.1268467139999999);
\coordinate [circle, fill=black, inner sep=0.7pt, label=360: {$\bar{y}_2$}] (y2) at (1.3280933615799932,1.7162109448380034);
\coordinate [circle, fill=black, inner sep=0.7pt, label=180: {$\bar{y}_3^*$}] (y3) at (-0.4838237980000031,9);%(-0.4838237980000031,13.040693192212995);
\coordinate [circle, fill=black, inner sep=0.7pt, label=180: {$\bar{y}_1^*$}] (y1) at (-0.4838237980000031,-5);%(-0.4838237980000031,-9.60827130253699);
\end{scriptsize}
\draw[dash pattern=on 1pt off 1pt] (-0.4838237980000031,13.040693192212995) -- (-0.4838237980000031,-9.60827130253699);
\draw[dash pattern=on 1pt off 1pt] (1.3280933615799932,1.7162109448380034)-- (-0.4838237980000031,13.040693192212995);
\draw[dash pattern=on 1pt off 1pt] (1.3280933615799932,1.7162109448380034)-- (-0.4838237980000031,-9.60827130253699);
\draw (p)--(y2);
\draw[color=red] (y10)--(p)--(y30);
\draw[color=blue] (y10)--(y2)--(y30);
\draw[color=blue, dash pattern=on 3pt off 3pt] (y1)--(p)--(y3);
\draw[color=red, dash pattern=on 3pt off 3pt] (y1)--(y2)--(y3);
\end{tikzpicture}
\end{center}
\caption{We assume that the path $\bar{y}_1^0\bar{y}_2\bar{y}_3^0$ is longer than the path $\bar{y}_1^0\bar{p}\bar{y}_3^0$. If the angle at $\bar{y}_2$ is not straight, we extend some lines and as points further from $\bar{p}$ are more on the left, we find some $t_1^*$ and $t_3^*$ such that $\bar{y}_1^*\bar{p}\bar{y}_3^*$ are collinear. But then the path $\bar{y}_1^0\bar{y}_2\bar{y}_3^0$ is shorter than $\bar{y}_1^0\bar{p}\bar{y}_3^0$, which then yields a contradiction to $\gamma$ being maximising.}
\label{fig: stacking_mainargument}
\end{figure}

\underline{Claim}: For each $t_1$ and $t_3$, equation (\ref{pop:compStack:eq:Convex}) has equality.

\underline{Then}: The comparison situation is straight, i.e.\ $\bar{y}_1,\bar{y}_2,\bar{y}_3$ are collinear, and by triangle equality along straight lines we have $\tau(y_1,y_3)=\tau(y_1,y_2)+\tau(y_2,y_3)=\bar{\tau}(\bar{y}_1,\bar{y}_2)+\bar{\tau}(\bar{y}_2,\bar{y}_3)=\bar{\tau}(\bar{y}_1,\bar{y}_3)$, i.e., $\bar{p},\bar{y}_1,\bar{y}_3$ is a comparison triangle for $p,y_1,y_3$.

\underline{Proof}: We indirectly assume there is a $t_1^0,t_3^0$ such that the comparison situation $\bar{\Delta}_{12},\bar{\Delta}_{23}$ is strictly convex. We draw this situation such that $\bar{p}=0$, $\bar{y}_2$ is to the right of the $t$-axis, such that the side $\bar{y}_2\bar{y}_3^0$ slopes to the left and $\bar{y}_1^0\bar{y}_2$ slopes to the right (i.e., $\frac{x(\bar{y}_3^0-\bar{y}_2)}{t(\bar{y}_3^0-\bar{y}_2)}<0$, $\frac{x(\bar{y}_2-\bar{y}_1^0)}{t(\bar{y}_2-\bar{y}_1^0)}>0$). When we vary $t_1$ and $t_3$, the comparison situation is chosen such that $\bar{p}$ and $\bar{y}_2$ stay fixed. As the comparison angle $\tilde{\ma}_{y_2}(p,\gamma(t_3))$ is monotonously increasing in $t_3$, we get that for $t_3\geq t_3^0$ the slope $\frac{x(\bar{y}_3-\bar{y}_2)}{t(\bar{y}_3-\bar{y}_2)}$ is smaller than the slope of $\bar{y}_2\bar{y}_3$. In particular, $\bar{y}_3$ lies to the left of the extension of the side $\bar{y}_2\bar{y}_3^0$. Thus, for large enough $t_3$, $\bar{y}_3$ lies to the left of the $t$-axis, and for a certain $t_3^*$ $\bar{y}_3^*$ lies on it. Similarly, we find a $t_1^*$ such that $\bar{y}_1^*$ lies on the $t$-axis. See Figure \ref{fig: stacking_mainargument} for a visualisation of the construction. But then 
\begin{equation*}
\tau(y_1^*,p)+\tau(p,y_3^*)=\bar{\tau}(\bar{y}_1^*,\bar{p})+\bar{\tau}(\bar{p},\bar{y}_3^*)> \bar{\tau}(\bar{y}_1^*,\bar{y}_2)+\bar{\tau}(\bar{y}_2,\bar{y}_3^*)=\tau(y_1^*,y_2)+\tau(y_2,y_3^*)=\tau(y_1^*,y_3^*)
\end{equation*}
in contradiction to the reverse triangle inequality. Thus we get the claim.
\end{proof}

\begin{remark}
The proof of the statement can also be used in more general situations: Let $X$ be a locally timelike geodesically connected Lorentzian pre-length space with local timelike curvature bounded below by $0$, and let $\gamma$ be a (possibly extendible) timelike maximiser. Let $p\ll x=\gamma(t_2)$ be points in a comparison neighbourhood. Assume the statement is not true (with certain $t_1<t_2<t_3$). Then we get $\omega_1=\tilde{\ma}_{y_2}(p,y_1)$, $\omega_3=\tilde{\ma}_{y_2}(p,y_3)$. Let $d_1+d_2=\omega_3-\omega_1$ and $e=\tau(p,x)\sinh(d_1+\omega_1)$. Then for the parameters $t_1= -\frac{e}{\sinh(d_1)}$ and $t_3= \frac{e}{\sinh(d_3)}$ we have one of the following:
\begin{itemize}
\item $\gamma$ is not defined at/inextendible to one of them,
\item $\gamma$ stops being distance realising between $t_1^*$ and $t_3^*$ (and we have found a longer curve),
\item there is no curvature comparison neighbourhood containing both $\gamma(t_i^*)$ ($i=1,3$).
\end{itemize}

For this, realise the above with $\ma_{\bar{y}_2}(\partial_t,\bar{y}_3)=d_3$ and $\ma_{\bar{y}_2}(\partial_t,\bar{y}_1)=d_1$. Then $e$ is the difference of the $x$-coordinates of $p$ and $x$.

Note the similarity of this result and classical arguments for conjugate points along geodesics in Riemannian/Lorentzian geometry.
\end{remark}

\begin{Proposition}[Angle = comparison angle]
\label{Proposition: angle=comparisonangle}
Let $X$ be a timelike geodesically connected \LpLS with global timelike curvature bounded below by $0$ and $\gamma:\R\to X$ be a complete timelike line and $x:=\gamma(t_0)$ a point on it. We split $\gamma$ into the future part $\gamma_+=\gamma|_{[t_0,+\infty)}$ and the past part $\gamma_-=\gamma|_{(-\infty,t_0]}$. Let $p\in X$ be a point not on $\gamma$ with $x$ and $p$ timelike related and $\alpha:x\leadsto p$ a connecting distance realiser. Then for all $s\neq t$ such that $p$ and $\gamma(s)$ are timelike related, we have:
\begin{equation*}
\tilde{\ma}_x(p,\gamma(s))=\ma_x(\alpha,\gamma_+)=\ma_x(\alpha,\gamma_-)\,,
\end{equation*}
i.e.\ the comparison angle is equal to the angle, and the same in both directions.
\begin{figure}
\begin{center}
\begin{tikzpicture}
\begin{scriptsize}
\coordinate [circle, fill=black, inner sep=0.7pt, label=360: {$x$}] (x) at (0,0);
\coordinate [circle, fill=black, inner sep=0.7pt, label=360: {$p$}] (p) at (-0.5,-1);
\coordinate (gplus) at (0,2);
\coordinate (gminus) at (0,-2);
\coordinate [inner sep=0.7pt, label=360: {$\gamma_+$}] (gp) at (0,1.5);
\coordinate [inner sep=0.7pt, label=360: {$\gamma_-$}] (gm) at (0,-1.5);
\end{scriptsize}
\draw (gminus)--(gplus);
\draw (p)--(x);
\draw pic[draw,angle radius={1.5*0.4cm},color=orange]{angle = p--x--gminus} pic[draw,angle radius={0.4cm},color=orange]{angle = gplus--x--p};
\end{tikzpicture}
\end{center}
\caption{Illustration: These angles are the same, and have the same value as if they are considered as comparison angles.}
\label{fig: stacking_angles}
\end{figure}
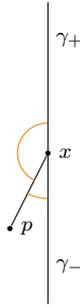
\end{Proposition}
\begin{proof}
First, we check that the comparison angle $\tilde{\ma}_x(p,\gamma(s))$ is constant in $s$: For $s_1$ and $s_2$ for which this is defined, we have three parameters on $\gamma$ involved: $s_1,s_2,t_0$. The previous result (Proposition \ref{Proposition: stackingprinciple}) tells us we can construct a comparison situation for all three triangles at once. As the comparison situations have the comparison angle $\tilde{\ma}_x(p,\gamma(s_1))$ resp.\ $\tilde{\ma}_x(p,\gamma(s_2))$ as the angle in $\bar{x}$, they are equal.

Now we look at the angle $\ma_x(\alpha,\gamma_\pm)$: We assume $p\ll x$. We already know $\tilde{\ma}_x(\alpha(s),\gamma(t))$ is constant in $t$. We now have to look at its dependence on $s$: By Theorem \ref{Theorem: equivalencecurvandmonotonicity}, $\tilde{\ma}_x^\mathrm{S}(\alpha(s),\gamma(t))$ is monotonously increasing in $s$. Note now that for $t<t_0$, the sign of this angle is $\sigma=-1$, and for $t>t_0$, the  sign of this angle is $\sigma=+1$. So choose some $t_-<t_0$ and $t_+>t_0$ for which all the necessary angles exist (i.e.\ $\gamma(t_-)\ll p$ and $t_+>t_0$) we have that $\tilde{\ma}_x(\alpha(s),\gamma(t_-))=\tilde{\ma}_x(\alpha(s),\gamma(t_+))$ is both a monotonously decreasing and increasing function in $s$. Thus it is constant, and $\tilde{\ma}_x(\alpha(s),\gamma(t))=\ma_x(\alpha,\gamma_-)=\ma_x(\alpha,\gamma_+)$, which includes the desired equalities.
\end{proof}

\begin{Corollary}\label{Corollary: sides_equal}
Let $X$ be a timelike geodesically connected, globally causally closed Lorentzian pre-length space with global timelike curvature bounded below by $0$ and let $\gamma:\R \to X$ be a complete timelike line. Then for any point $p\in X$ and two points $x_1=\gamma(t_1)$ and $x_2=\gamma(t_2)$ which are timelike related to $p$, we get a timelike triangle $\Delta=\Delta(x_1,x_2,p)$. Let $q_1,q_2$ be any points on $\Delta$, one of them lying on the side $x_1x_2$. We form a comparison triangle $\bar{\Delta}$ for $\Delta$ and find comparison points $\bar{q}_1,\bar{q}_2$ for $q_1,q_2$. Then $q_1\leq q_2$ if and only if $\bar{q}_1\leq\bar{q}_2$ and $\tau(q_1,q_2)=\bar{\tau}(\bar{q}_1,\bar{q}_2)$.
\end{Corollary}
\begin{proof}
As global curvature is bounded below by $0$, we get that $\tau(q_1,q_2)\leq\bar{\tau}(\bar{q}_1,\bar{q}_2)$, and a continuity argument then shows that $q_1\leq q_2$ implies $\bar{q}_1\leq\bar{q}_2$.

For the other direction, we distinguish which sides the $q_i$ lie on: Note we assumed one of them is on the side $x_1x_2$. We only prove the case where $q_1$ is on the side $\alpha$ connecting $x_1x_2$ (say $q_1=\alpha(s)$) and $q_2$ is on the side $\beta$ connecting $x_1,p$ (say $q_2=\beta(t)$), the proof of the other cases is easily adapted.

We can now consider the hinge $(\alpha|_{[0,s]},\beta|_{[0,t]})$, which has base-point $x_1$ and two tips at $q_1,q_2$. By the previous Proposition \ref{Proposition: angle=comparisonangle}, we get that $\omega:=\ma_{x_1}(\alpha,\beta)=\tilde{\omega}:=\tilde{\ma}_{x_1}(p,x_2)$. We now have two comparison situations: $\bar{\Delta}$ has angle $\tilde{\ma}_{x_1}(p,x_2)$ at $x_1$, and a comparison hinge $(\tilde{\alpha},\tilde{\beta})$ with tips $\tilde{q}_1,\tilde{q}_2$. But note that the angles and distances at $\bar{x}_1$ resp.\ $\tilde{x}_1$ are the same, so we have that $\bar{\tau}(\bar{q}_1,\bar{q}_2)=\bar{\tau}(\tilde{q}_1,\tilde{q}_2)$. But for the comparison hinge, \cite[Cor.\ 4.11]{beran2022angles} gives that $\tau(q_1,q_2)\geq\bar{\tau}(\tilde{q}_1,\tilde{q}_2)=\bar{\tau}(\bar{q}_1,\bar{q}_2)$. In total, this gives that $\tau(q_1,q_2)=\bar{\tau}(\bar{q}_1,\bar{q}_2)$, as desired. A continuity argument then gives that $q_1\leq q_2$ if and only if $\bar{q}_1\leq\bar{q}_2$.
\end{proof}

We return to the situation of the splitting theorem. Let $\gamma:\R \to X$ be the complete timelike line whose existence we assume, and let it be in any locally Lipschitz parametrisation (e.g.\ parametrisation by $d$-arclength) defined on $\R$. We call any co-ray to any of its past or future directed subrays a co-ray to the line $\gamma$, so in particular, we can construct past and future directed co-rays from all points on $I(\gamma):=I^+(\gamma) \cap I^-(\gamma)$. The next result shows that the timelike co-ray condition holds on $I(\gamma)$, i.e.\ all co-rays from all points in $I(\gamma)$ are timelike. For our purposes it would be sufficient to show this for asymptotes, since we only work with those in the proof.

\begin{Proposition}
\label{proposition: TCRCholdsonI(gamma)}
The timelike co-ray condition holds on $I(\gamma)$.
\begin{proof}
Suppose there is a point $x \in I(\gamma)$ such that the TCRC does not hold at $x$, so w.l.o.g.\ there is a sequence $x_n \to x$, $r_n \to \infty$ and maximal timelike curves $\sigma_n$ from $x_n$ to $\gamma(r_n)$ such that $\sigma_n$ converge to a future directed null ray $\sigma$. Choose some $q \in \gamma \cap I^-(x)$ and let $\mu_n$ be maximal timelike curves from $q$ to $x_n$ (assuming $n$ to be large enough, $q\ll x_n$). Suitably pre- and post-composing $\mu_n$ and then applying the limit curve theorem, it is easily seen that (up to a choice of subsequence) the $\mu_n$ converge locally uniformly to a maximising limit causal curve $\mu$ from $q$ to $x$. $\mu$ is timelike since $q \ll x$. Denote by $\gamma_n$ the piece of $\gamma$ that runs between $q$ and $\gamma(r_n)$. Set $a_n:=L_{\tau}(\mu_n)$, $b_n:=L_{\tau}(\sigma_n)$ and $c_n:=L_{\tau}(\gamma_n)$. Let $\beta_n:=\ma_q(\mu_n,\gamma_n)$ and $\theta_n:=\ma_{x_n}(\mu_n,\sigma_n)$. Then $a_n \to a:=\tau(q,x)$ and by the continuity of angles (cf.\ Proposition \ref{proposition: continuityofangles}), $\beta_n \to \beta$, where $\beta$ is the angle between $\mu$ and $\gamma$. Consider the comparison triangle $(\overline{\mu}_n,\overline{\sigma}_n,\overline{\gamma}_n)$ in $\R^{1,1}$, then the angle $\overline{\beta}_n$ between $\overline{\mu}_n$ and $\overline{\gamma}_n$ satisfies $\overline{\beta}_n \leq \beta_n$ and similarly $\overline{\theta}_n \geq \theta_n$. Since $\beta_n \to \beta$, there is some $C > 0$ such that $\overline{\beta}_n \leq C$ for all $n$. The hyperbolic law of cosines in $\R^{1,1}$ (see Lemma \ref{Lemma: hyperboliclawofcosines}) gives
\begin{align*}
    &b_n^2 = a_n^2 + c_n^2 - 2a_n c_n \cosh(\overline{\beta}_n),\\
    &c_n^2 = a_n^2 + b_n^2 + 2a_n b_n \cosh(\overline{\theta}_n).
\end{align*}
Using these two equations and solving for $\cosh(\overline{\theta}_n)$, we get 
\begin{align*}
    \cosh(\overline{\theta}_n) = - \frac{a_n}{b_n} + \frac{c_n}{b_n} \cosh(\overline{\beta}_n).
\end{align*}

By the initial equation for $b_n$, it is easy to see that $b_n/c_n \to 1$ and $b_n \to \infty$ (using that $a_n \to a$ and $\overline{\beta}_n \leq C$), hence there is some constant $\tilde{C} > 0$ such that $\cosh(\overline{\theta}_n) \leq \tilde{C}$ and thus also $\theta_n \leq C'$ for some constant $C'$. Using the monotonicity condition, we get that $\theta_n=\ma_{x_n}(\sigma_n,\mu_n)\geq \overline{\measuredangle}_{x_n}(\sigma_n(s),\mu_n(t)))$ for each $s$ and $t$, so this is bounded too. We will see that this is incompatible with $\sigma_n$ "getting more and more null": Note that we get the following estimate for some constant $C''$:
\begin{align*}
    C'' \geq \cosh( \overline{\measuredangle}_{x_n}(\sigma_n(s),\mu_n(t))) = \frac{\tau(\mu_n(t),\sigma_n(s))^2 - \tau(\mu_n(t),x_n)^2 - \tau(x_n,\sigma_n(s))^2}{2\tau(x_n,\sigma_n(s))\tau(\mu_n(t),x_n)}.
\end{align*}
Since the denominator goes to $0$ for $n \to \infty$ (as $\tau(x_n,\sigma_n(s)) \to \tau(x,\sigma(s)) = 0$ and $\tau(\mu_n(t),x_n) \to \tau(\mu(t),x) > 0$), the numerator has to go to $0$ as well, in particular this means
\begin{align*}
    \tau(\mu(t),\sigma(s)) = \tau(\mu(t),x)
\end{align*}
for all $s,t$. This implies that running along $\mu$ from $\mu(t)$ to $x$ and then along $\sigma$ to $\sigma(s)$ gives a maximiser, but this curve has a timelike and a null piece, a contradiction to Theorem \ref{theorem: causalcharofmaximizers}.
\end{proof}
\end{Proposition}

Next, we show that any asymptote to $\gamma$ (which we now know to be timelike) is complete.

\begin{Proposition} Let $x \in I(\gamma)$ and let $\eta$ be a timelike asymptote to $\gamma$ at $x$. Then $L_{\tau}(\eta) = \infty$.
\begin{proof}
W.l.o.g. let $\eta$ be future-directed. Suppose $L:=L_{\tau}(\eta) < \infty$. By construction, $\eta:[0,\infty) \to X$ arises as a locally uniform limit of timelike maximisers $\eta_n:[0,a_n] \to X$ from $x$ to $\gamma(t_n)$, where $t_n \to \infty$. By continuity of angles (see Proposition \ref{proposition: continuityofangles}), $\measuredangle_x(\eta,\eta_n) \to 0$ for $n \to \infty$. Let $\varepsilon > 0$ and let $N \in \mathbb{N}$ be such that for $n \geq N$, $\measuredangle_x(\eta,\eta_n) < \varepsilon$ (in fact, any finite bound would suffice here, we do not necessarily need an arbitrarily small one). Since $\tau(x,\gamma(t_n)) \to \infty$, we may assume that $\tau(x,\gamma(t_n)) \geq 3L\cosh(\varepsilon)$ (upon possibly choosing a larger $N$). Now note that for any $t_n$ with $n \geq N$, $\partial J^-(\gamma(t_n)) \cap \eta = (J^-(\gamma(t_n))\setminus I^-(\gamma(t_n))) \cap \eta$ is non-empty: Certainly, $x \in I^-(\gamma(t_n))$, so if this intersection were empty, then $\eta$ would be imprisoned in the compact set $J^+(x) \cap J^-(\gamma(t_n))$, which cannot happen. So we find a point $y_0 \in \eta$ that is null-related to $\gamma(t_n)$, i.e. $y_0 < \gamma(t_n)$ and $\tau(y_0,\gamma(t_n)) = 0$. By continuity, there is $y$ on $\eta$ near $y_0$ such that $0 < \tau(y,\gamma(t_n)) < 3L\cosh(\varepsilon)/2$. Let now $\nu$ be the part of $\eta$ from $x$ to $y$ with length $L_{\tau}(\nu) =: a$, $\sigma$ a timelike maximiser from $y$ to $\gamma(t_n)$ with length $b:=L_{\tau}(\sigma) = \tau(y,\gamma(t_n))$. Moreover, we write $c:=L_{\tau}(\eta_n)$. Then $(\nu,\sigma,\eta_n)$ forms a timelike triangle with angle $\beta:=\measuredangle_x(\eta,\eta_n) < \varepsilon$ at $x$, consider a corresponding comparison triangle $(\overline{\nu},\overline{\sigma},\overline{\eta}_n)$ with angle $\overline{\beta}$ at $\overline{x}$. By the law of cosines, we get
\begin{align}
    b^2 = a^2 + c^2 - 2ac \cosh(\overline{\beta}).
\end{align}
Our curvature assumption gives that $\overline{\beta} \leq \beta < \varepsilon$. Moreover, as we have argued, $c \geq 3L\cosh(\varepsilon)$ and $b < c/2$ and $a < L$. Inserting all of this, we get
\begin{align*}
    b^2 &\geq a^2 + c^2 - 2ac\cosh(\varepsilon) \geq a^2 + c^2 - 2Lc \cosh(\varepsilon) \geq c^2 - 2Lc \cosh(\varepsilon) \\
    &= c^2 \left(1 - \frac{2L\cosh(\varepsilon)}{c}\right) \geq c^2/3,
\end{align*}
contradicting $b < c/2$.
\end{proof}
\end{Proposition}

To conclude this subsection, we show that any future directed and any past directed (timelike) asymptote to $\gamma$ from a common point fit together to give a timelike line.

\begin{Proposition}[Asymptotic lines]
\label{Proposition: asymptoticlines}
Let $p \in I(\gamma)$ and consider any future and past rays $\sigma^+:[0,\infty) \to X$ and $\sigma^-:(-\infty,0] \to X$ from $p$ to $\gamma$. Then $\sigma :=\sigma^- \sigma^+: \R \to X$ is a complete timelike line.
\begin{proof}
Let $\sigma_n^+$ and $\sigma_n^-$ be two sequences of timelike maximisers from $p$ to $\gamma(r_n)$ and $\gamma(-r_n)$, respectively, such that $\sigma^+$ and $\sigma^-$ arise as limits of these sequences as $r_n \to \infty$. To show that $\sigma$ is a line, it is sufficient to show that for any $t > 0$, $\tau(\sigma(-t),\sigma(t)) = L_{\tau}(\sigma|_{[-t,t]})$. To see this, let $q_+:=\sigma(t)$ and $q_-:=\sigma(-t)$, and $q_+^n:=\sigma_n^+(t)$, $q_-^n:=\sigma_n^-(-t)$. Then $q_{\pm} = \lim_n q_{\pm}^n$. Consider the triangle going from $x$ via $\sigma_n^+$ to the endpoint of $\sigma_n^+$ , then following $\gamma$ down to the endpoint of $\sigma_n^-$, and finally running the latter up to $x$ again. Consider a comparison triangle with points $\overline{q_{\pm}^n}$ corresponding to $q_{\pm}^n$. Sending $n \to \infty$, we see that the stacked comparison triangles in $\R^{1,1}$ converge to a vertical line (here we use Proposition \ref{Proposition: stackingprinciple} and the completeness of $\gamma$), hence our curvature bound gives
\begin{align*}
    \tau(q_-,q_+) = \lim_{n \to \infty} \tau(q_-^n,q_+^n) \leq \lim_{n \to \infty} \overline{\tau}(\overline{q_-^n},\overline{q_+^n}) = \lim_n L_{\tau}(\sigma_n^-\sigma_n^+|_{[-t,t]})= L_{\tau}(\sigma|_{[-t,t]}),
\end{align*}
which is what we wanted to show, as the other inequality is trivial.
\end{proof}
\end{Proposition}

\subsection{Parallel lines}
\label{subsec:lines:parallel}

This subsection introduces the notion of parallelity for complete timelike lines. This approach is better suited for the synthetic case as it circumvents the analysis of Busemann functions. In the following, we call a map $f:Y_1 \to Y_2$ between Lorentzian pre-length spaces $(Y_1,d_1,\ll_1,\leq_1,\tau_1)$ and $(Y_2,d_2,\ll_2,\leq_2,\tau_2)$ $\tau$-\emph{preserving} if for all $p,q \in Y_1$, $\tau_1(p,q) = \tau_2(f(p),f(q))$, and we call $f$ $\leq$-\emph{preserving} if $p \leq_1 q$ if and only if $f(p) \leq_2 f(q)$.

\begin{definition}[Parallel lines]
\label{definition: parallellines}
Let $\alpha,\beta$ be two complete timelike lines in a Lorentzian pre-length space $X$ defined on open intervals, w.l.o.g.\ on $\R$. They are called \emph{parallel} if there exists a $\tau$- and $\leq$-preserving map $f:(\alpha(\R)\cup \beta(\R))\to \R^{1,1}$ such that $f(\alpha(\R))$ and $f(\beta(\R))$ are parallel timelike lines in Minkowski (in the sense of parallel lines in affine spaces). We call such a map $f$ a \emph{parallel realisation} of $\alpha$ and $\beta$.
\end{definition}

\begin{remark}
Note that by post-composing this by an isometry of Minkowski space, we can always achieve that $f(\alpha(\R))=\{(t,0):t\in\R\}\subseteq\R^{1,1}$ and $f(\beta(\R))=\{(t,c):t\in\R\}$ for some $c\geq0$.
In this form, if $\alpha$ and $\beta$ are parallel and both parametrized by $\tau$-arclength (this is possible if $\tau$ is continuous and $\tau(x,x) = 0$ for all $x \in X$, cf.\ Lemma \ref{Lemma: tauarclengthparametrizations}), we get that $f(\alpha(t))=(t+a,0)$ and $f(\beta(t))=(t+b,c)$. By doing a shift in Minkowski, we can make $a$ to be $0$ (changing $b$ to $b-a$).
\end{remark}

\begin{Lemma}[Properties of parallel realisations]
\label{Lemma: propertiesofparallelrealization}
Let $X$ be a strongly causal Lorentzian pre-length space with continuous time separation, let $\alpha,\beta:\R \to X$ be two complete timelike lines and let $f: \alpha(\R) \cup \beta(\R) \to \R^{1,1}$ be a parallel realisation. Then $f$ is a homeomorphism onto its image.
\begin{proof}
We first show that $f$ is injective. Certainly, due to the fact that $f$ preserves time separations, $f|_{\alpha(\R)}$ and $f|_{\beta(\R)}$ are injective: Indeed suppose that e.g.\ $f(\alpha(t)) = f(\alpha(s)) = (r,x)$ for some $s < t$, then $0<\tau(\alpha(s),\alpha(t)) = \overline{\tau}((r,x),(r,x)) = 0$, a contradiction. Now suppose that the lines $f(\alpha(\R))$ and $f(\beta(\R))$ in $\R^{1,1}$ are different, then we are done. Otherwise, by nature of lines in Minkowski space, they have to be equal if they intersect. In that case, each point on that line is reached by one point on $\alpha$ and one point on $\beta$ via $f$. So suppose that $(r,x) = f(\alpha(t)) = f(\beta(s))$. Since $f$ is $\tau$-preserving and $\tau$ is continuous, for each $\varepsilon > 0$ there are $\delta,\tilde{\delta} > 0$ such that $f(\beta(s + \delta)) = (r+\varepsilon,x)$ and $f(\beta(s-\tilde{\delta})) = (r-\varepsilon,x)$. But this means that $\alpha(t) \in \bigcap_{\delta,\tilde{\delta} \downarrow 0} I(\beta(s-\tilde{\delta}),\beta(s+\delta))$. By strong causality, these sets are a neighbourhood basis at $\beta(s)$, hence $\alpha(t) = \beta(s)$. This concludes the proof that $f$ is injective.

Clearly, $\alpha(\R) \cup \beta(\R)$ is a strongly causal Lorentzian pre-length space with the restriction of the structure of $X$, and similarly for their images in Minkowski space. Moreover, $f$ maps the timelike diamonds in $\alpha(\R) \cup \beta(\R)$ to the timelike diamonds in $f(\alpha(\R)) \cup f(\beta(\R))$. As these form topological bases due to strong causality, we conclude that $f$ is a homeomorphism onto its image.
\end{proof}
\end{Lemma}

Note that in the above setting, any parallel realisation $f$ gives us a $\tau$-arclength parametrisation of the parallel lines $\alpha,\beta$: Let $\tilde{\alpha}$ be a $\tau$-arclength parametrisation of the line $f(\alpha(\R))$ in Minkowski, then $f^{-1}\circ\tilde{\alpha}$ is $\tau$-arclength parametrised, similarly for $\beta$.

\begin{definition}[Synchronised parallel lines]
\label{definition: syncparallellines}
Let $X$ be a Lorentzian pre-length space with continuous time separation $\tau$ satisfying $\tau(x,x) = 0$ for all $x\in X$. Let $\alpha:\R\to X$, $\beta:\R\to X$ be two $\tau$-arclength parametrised, parallel and complete timelike lines. They are called \emph{synchronised parallel}\footnote{This agrees with the usual notion of synchronising clocks in the sense that this is the parametrisation coming from synchronising the clocks on the particles $\alpha$ and $\beta$ and parametrising $\alpha$ and $\beta$ by their clock values.} if the parallel realisation $f:(\alpha(\R)\cup \beta(\R))\to \R^{1,1}$ can be chosen to be of the form $f(\alpha(t))=(t,0)$ and $f(\beta(t))=(t,c)$ for some $c\geq 0$. The constant $c$ is a well-defined property of $(\alpha,\beta)$ and is called the \emph{distance} of the parallel lines. For any two parallel lines $\alpha$ and $\beta$ parametrised by $\tau$-arclength, one can shift $\beta$ such that $(\alpha,\beta\circ(t\mapsto t-a))$ is synchronised parallel.
\end{definition}

\begin{Lemma}[The $c$-criterion for parallel lines]\label{lem-c-cond}
Let $X$ be a \LpLS with continuous time separation $\tau$ satisfying $\tau(x,x) = 0$ for all $x \in X$. Let $\alpha:\R\to X$ and $\beta:\R\to X$ be two $\tau$-arclength parametrised complete timelike lines. We define the following partial functions:
\begin{itemize}
\item $c_{\alpha\beta}(s,t)=\sqrt{(t-s)^2-\tau(\alpha(s),\beta(t))^2} (\in \mathbb{C})$ if $\alpha(s)\leq\beta(t)$ (otherwise undefined),
\item $c_{\beta\alpha}(s,t)=\sqrt{(s-t)^2-\tau(\beta(t),\alpha(s))^2} (\in \mathbb{C})$ if $\beta(t)\leq\alpha(s)$ (otherwise undefined), 
\item $c_{\alpha+}^N(s)=\inf\{t-s:\alpha(s)\leq\beta(t)\}$,
\item $c_{\beta+}^N(s)=\inf\{t-s:\beta(s)\leq\alpha(t)\}$.  
\end{itemize}
These are all constant where defined, have the same value and the infima are minima if and only if $\alpha$ and $\beta$ are synchronised parallel. This constant $c$ is the distance between $\alpha$ and $\beta$.

If $X$ is additionally globally causally closed and $\alpha \cap I^+(\beta) \neq \emptyset$ and $\beta \cap I^+(\alpha) \neq \emptyset$, then the infima in $c_{\alpha+}^N$ and $c_{\beta+}^N$ are automatically minima and the conditions that $c_{\alpha+}^N$ and $c_{\beta+}^N$ are constant and have the same value are automatically satisfied if $c_{\alpha
\beta}$ and $c_{\beta\alpha}$ are constant and have the same value. In that case, all of the constants agree.
\end{Lemma}
\begin{proof}
We define $f$ as in the definition of synchronised parallel with distance $c$: $f(\alpha(s)):=(s,0)$, $f(\beta(s)):=(s,c)$. We will prove that $f$ is $\leq$-preserving if and only if $c_{\alpha+}^N$ and  $c_{\beta+}^N$ are constantly $c$, and under the assumption that this is the case $f$ is $\tau$-preserving if and only if $c_{\alpha\beta}$ and $c_{\beta\alpha}$ are constantly $c$ wherever defined.

As $\alpha$ and $\beta$ are future directed $\tau$-arclength parametrised timelike lines, it is clear that $f$ is $\leq$- and $\tau$-preserving along $\alpha$ and along $\beta$, i.e.\ we only have to check the conditions on $f(\alpha(s))\leq f(\beta(t))$, their $\tau$-distance and the same for $\alpha,\beta$ reversed.

So first, for $\leq$-preserving: We know that $f(\alpha(s))\leq f(\beta(t))\Leftrightarrow (s,0)\leq (t,c)\Leftrightarrow (t-s)\geq c$. On the other hand, $\alpha(s)\leq \beta(t) \Leftrightarrow t-s\geq c_{\alpha+}^N(s)$ if this is a minimum and $\alpha(s)\leq \beta(t) \Leftrightarrow t-s> c_{\alpha+}^N(s)$ if it is only an infimum. In particular, these conditions are the same precisely when $c_{\alpha+}^N(s)=c$ and it is a minimum. Thus, $\alpha(s)\leq\beta(t)\Leftrightarrow f(\alpha(s))\leq f(\beta(t))$ if and only if $c_{\alpha+}^N$ is constantly $c$ and is always a minimum. Analogously, we get that $\beta(s)\leq\alpha(t)\Leftrightarrow f(\beta(s))\leq f(\alpha(t))$ if and only if $c_{\beta+}^N$ is constantly $c$ and is always a minimum.

For $\tau$-preserving, we assume that $f$ is already $\leq$-preserving. On the Minkowski side, we have: 
$\bar{\tau}(f(\alpha(s)),f(\beta(t)))=\sqrt{(t-s)^2-c^2}$ (if $f(\alpha(s))\leq f(\beta(t))$). On the $X$ side, we solve the defining equation for $c_{\alpha\beta}$ for $\tau(\alpha(s),\beta(t))$ to get: $\tau(\alpha(s),\beta(t))=\sqrt{(t-s)^2-c_{\alpha\beta}(s,t)^2}$ (if $\alpha(s)\leq\beta(t)$). Now note that as $f$ is $\leq$-preserving, the side conditions when these equations can be applied match up (if they are not applied, $\tau(\alpha(s),\beta(t))=0=\bar{\tau}(f(\alpha(s)),f(\beta(t)))$ anyway). If the side conditions are satisfied, we note that the equation differ only in one place: $\tau(\alpha(s),\beta(t))$ uses $c_{\alpha\beta}(s,t)$ and $\bar{\tau}(f(\alpha(s)),f(\beta(t)))$ uses $c$. Thus, we easily see that $\tau(\alpha(s),\beta(t))=\bar{\tau}(f(\alpha(s)),f(\beta(t)))$ if and only if $c_{\alpha\beta}(s,t)=c$.

Analogously, we get that $\tau(\beta(t),\alpha(s))=\bar{\tau}(f(\beta(t)),f(\alpha(s)))$ if and only if $c_{\beta\alpha}(s,t)=c$.

For the note, fix $s$. If $X$ satisfies the additional assumptions, $\{t:\alpha(s)\leq\beta(t)\}$ is closed and bounded (as is easily seen from the fact that $c_{\alpha\beta} \in \mathbb{R}$ when defined due to reverse triangle inequality), thus it has a minimum automatically. We look at $t$ such that $\alpha(s)\ll\beta(t)$, the set of them is $(t_s,+\infty)=\beta^{-1}(I^+(\alpha(s)))$. We get that $\tau(\alpha(s),\beta(t))\to0$ as $t\searrow t_s$. We transform the equation $c=\sqrt{(t-s)^2-\tau(\alpha(s),\beta(t))^2}$ to $\tau(\alpha(s),\beta(t))=\sqrt{(t-s)^2-c^2}$. By continuity of $\tau$, this still holds in the limit $t\searrow t_s$, i.e.\ $\sqrt{(t_s-s)^2-c^2}=\tau(\alpha(s),\beta(t_s))=0$, so $t_s-s=c$ and $\{t:\alpha(s)\leq\beta(t)\}=[t_s,+\infty)$ as claimed.
\end{proof}

\begin{Lemma}[The strong causality trick]
Let $X$ be a strongly causal \LpLS with continuous $\tau$ and $\tau(x,x)=0$ for all $x \in X$. Let $\alpha,\beta:[0,b)\to\R$ be two $\tau$-arclength parametrised timelike distance realisers with $x:=\alpha(0)=\beta(0)$. Assume that for all $s,t$ such that $\alpha(s)$ and $\beta(t)$ are timelike related, the comparison angle $\tilde{\ma}_x(\alpha(s),\beta(t))=0$. Then $\alpha=\beta$.
\end{Lemma}
\begin{proof}
We set $f_+(s,t)=\tau(\alpha(s),\beta(t))$ and $f_-(s,t)=\tau(\beta(s),\alpha(t))$. They are both monotonically increasing in $t$ and continuous. We want to describe the set where $f_+>0$ resp.\ $f_->0$. We set $t_s^+=\inf \{t:f_+(s,t)>0\}$, then $\lim_{t\searrow t_s^+} f_+(s,t)=0$. Thus in the law of cosines (Lemma \ref{Lemma: hyperboliclawofcosines}), we get $\lim_{t\searrow t_s^+}\cosh(\tilde{\ma}_x(\alpha(s),\beta(t)))=1=\lim_{t\searrow t_s^+} \frac{s^2+t^2-f_+(s,t)}{2st}=\frac{s^2+(t_s^+)^2}{2st_s^+}$, so $s=t_s^+$. Analogously, we get $s=\inf \{t:f_-(s,t)>0\}$, in total $f_\pm(s,t)>0$ for $s<t$.

That means that whenever $s_-<t<s_+$ we have that $\alpha(t)\in I(\beta(s_-),\beta(s_+))$. By Lemma \ref{Lem:strongly_strongly_causal}, the $I(\beta(s_-),\beta(s_+))$ form a neighbourhood basis of the point $\beta(t)$, and $\alpha(t)$ is inside of all of these neighbourhoods. As $X$ is Hausdorff, we get that $\alpha(t)=\beta(t)$ for all $t\in(0,b)$, so $\alpha=\beta$.
\end{proof}

\begin{Lemma}[Parallel lines are unique]
\label{Lemma: parallellinesunique}
Let $X$ be a strongly causal, timelike geodesically connected \LpLS with continuous $\tau$, $\tau(x,x)=0$ for all $x \in X$. Suppose that $X$ has global timelike curvature bounded below by $0$, let $\alpha:\R\to X$ be a complete timelike line and $p\in X$ a point. Then there is (up to reparametrisation) at most one parallel line to $\alpha$ through $p$.
\end{Lemma}
\begin{proof}
We indirectly assume there are two parallel lines to $\alpha$ through $p$, namely $\beta:\R\to X$ and $\tilde{\beta}:\R\to X$. We assume all three curves are parametrised by $\tau$-arclength, such that $\beta(0)=p$ and in such a way that $\alpha,\beta$ are synchronised parallel. Then also $\alpha,\tilde{\beta}$ are parametrised in a way that they are synchronised parallel, and both $\alpha,\beta$ and $\alpha,\tilde{\beta}$ have the same distance: set $s_+=\inf \alpha^{-1}(I^+(p))$ and $s_-=\sup \alpha^{-1}(I^-(p))$ the boundary of the points on $\alpha$ timelike related to $p$. Then the distance of $\alpha,\beta$ and that of $\alpha,\tilde{\beta}$ can be calculated as $\frac{s_+-s_-}{2}$, so they are the same, and the necessary shift to make $\beta$ and $\tilde{\beta}$ synchronised parallel to $\alpha$ can be calculated by $\frac{s_++s_-}{2}$, so they are the same too. As $\alpha,\beta$ are assumed to be synchronised already, this shift is $0$.

We construct a comparison situation for all three lines at once: We choose the parallel realisation $f$ for $\alpha$ and $\beta$ given by $f(\alpha(s))=(s,0)$ and $f(\beta(s))=(s,c)$, and similarly we choose $\tilde{f}$ for $\alpha$ and $\tilde{\beta}$ given by $\tilde{f}(\tilde{\beta}(s))=(s,-c)$ and $\tilde{f}(\alpha(s)) = (s,0)$, i.e.\ we realise $\beta$ and $\tilde{\beta}$ on opposite sides of $\alpha$. 

We calculate the angle $\omega:=\ma_p(\beta|_{[0,+\infty)},\tilde{\beta}|_{[0,+\infty)})$: By Proposition \ref{Proposition: angle=comparisonangle}, this is equal to any comparison angle $\tilde{\ma}_p(\beta(s),\tilde{\beta}(t))$ as long as $\beta(s)\ll\tilde{\beta}(t)$ or conversely. We know that $\beta(s)\ll \alpha(s + c +\varepsilon)\ll \tilde{\beta}(s+2c+2\varepsilon)$ and set $t=s+2c+2\varepsilon$. We now look at the law of cosines to estimate the comparison angle $\omega=\tilde{\ma}_p(\beta(s),\tilde{\beta}(t))$: $\cosh(\omega)=\frac{s^2+t^2-\tau(\beta(s),\tilde{\beta}(t))^2}{2st} \leq \frac{2t^2}{2s^2}$ which converges to $1$ as $s\to +\infty$. In particular, $\omega=0$. We can now apply the strong causality trick to get $\beta=\tilde{\beta}$ (so they are even equal in this parametrisation), so there is only one line parallel to $\alpha$ through $p$.
\end{proof}

Let us now show for $X$ as in the splitting theorem that asymptotic lines to $\gamma$ constructed in Proposition \ref{Proposition: asymptoticlines} are parallel to $\gamma$.

\begin{Lemma}
\label{Lemma: asymptotestogammaareparalleltogamma}
Let $\alpha: \R \to X$ be a complete timelike asymptotic line to $\gamma$. Then $\alpha,\gamma$ are parallel.
\begin{proof}
By construction, $\alpha$ arises as the limit of timelike maximising segments $\alpha^+_n$ from $p:=\alpha(0)$ to $\gamma(r_n)$ and $\alpha^-_n$ from $p$ to $\gamma(-r_n)$, where $r_n \to \infty$. We will use the stacking principle (cf.\ Proposition \ref{Proposition: stackingprinciple}) to show that $\alpha$ and $\gamma$ are parallel. Indeed, the triangles $\Delta(\gamma(-r_n),p,\gamma(r_n))$ stack in $\R^{1,1}$, hence by orienting the side corresponding to the segment on $\gamma$ vertically (see Figure \ref{fig: stacking}), we may assume that for $n \to \infty$ the comparison triangles converge to two vertical lines in $\R^{1,1}$. 

\begin{figure}
\begin{center}
\begin{tikzpicture}
%line \gamma
\draw (0,-3.3) -- (0,3.3);

\begin{scriptsize}
%point p
\coordinate [circle, fill=black, inner sep=0.7pt, label=0: {$p$}] (p) at (1,0);
%label of \gamma
\coordinate [label=180: {$\gamma$}] (p0) at (0,0);
%label of n-vertices
\coordinate [circle, fill=black, inner sep=0.7pt, label=180: {$\gamma(r_n)$}] (p1) at (0,1.8);
\coordinate [circle, fill=black, inner sep=0.7pt, label=180: {$\gamma(-r_n)$}] (p2) at (0,-1.8);
%label of (n+1)-vertices
\coordinate [circle, fill=black, inner sep=0.7pt, label=180: {$\gamma(r_{n+1})$}] (p1) at (0,2.3);
\coordinate [circle, fill=black, inner sep=0.7pt, label=180: {$\gamma(-r_{n+1})$}] (p2) at (0,-2.3);
%label of \alpha segments
\coordinate [label=270: {$\alpha_n^+$}] (p0) at (0.3,1);
\coordinate [label=90: {$\alpha_n^-$}] (p0) at (0.3,-1);
\end{scriptsize}
%stacking triangles
\draw (0,-1.8) -- (1,0) -- (0,1.8);
\draw (0,-2.3) -- (1,0) -- (0,2.3);
\draw (0,-2.8) -- (1,0) -- (0,2.8);
%\draw [dashed] (0,-3.1) -- (1,0) -- (0,3.1);
\draw [dashed] (0,-3.2) -- (1,0) -- (0,3.2);
\end{tikzpicture}
\end{center}
\caption{Stacking comparison triangles.}
\label{fig: stacking}
\end{figure}
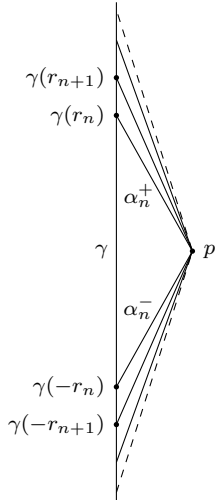

We now argue that the map sending $\alpha(\R)$ and $\gamma(\R)$ to the corresponding limit lines $\bar{\alpha}$ and $\bar{\gamma}$ is a parallel realisation. For any $t,s$, consider $\tau(\alpha(t),\gamma(s)) = \lim_n \tau(\alpha_n(t),\gamma(s))$. For $n$ so large that $-r_n < s < r_n$, $\gamma(s)$ is part of the triangle $\Delta(\gamma(-r_n),p,\gamma(r_n))$. From Corollary \ref{Corollary: sides_equal} we conclude that $\alpha_n(t) \ll \gamma(s)$ if and only if $\overline{\alpha}_n(t) \ll \overline{\gamma}(s)$, and $\tau(\alpha_n(t),\gamma(s)) = \overline{\tau}(\overline{\alpha}_n(t),\overline{\gamma}(s))$, where the bars denote the corresponding points on the comparison side. This shows that the realisation of $\alpha$ as $\bar{\alpha}$ and $\gamma$ as $\bar{\gamma}$ is a parallel realisation.
\end{proof}
\end{Lemma}

\begin{Lemma}[Shifting asymptotes] \label{Lemma: shifting_asymptotes}
Let $\alpha: \R \to X$ be a complete timelike asymptote to $\gamma$ through $\alpha(0)$. Then the asymptote to $\gamma$ through $\alpha(s)$ is $\alpha(\cdot-s)$ after parametrising both by $\tau$-arclength.
\end{Lemma}
\begin{proof}
Let $\tilde{\alpha}$ be the asymptotic line through $\alpha(s)$. Then the previous result shows that both $\gamma$ and $\alpha$ as well as $\gamma$ and $\tilde{\alpha}$ are parallel and they meet at $\alpha(s)=\tilde{\alpha}(0)$. Thus, as parallel lines are unique (see Lemma \ref{Lemma: parallellinesunique}), we have that (after synchronising) $\alpha=\tilde{\alpha}$. To get the correct parameters before shifting, we know $\alpha(s)=\tilde{\alpha}(0)$, which fixes the shift parameter.
\end{proof}

\begin{Corollary}[Asymptotes stay in $I(\gamma)$]
\label{Corollary: asymptotesstayinI(gamma)}
Asymptotic lines to $\gamma$ from points in $I(\gamma)$ stay in $I(\gamma)$.
\begin{proof}
Since asymptotic lines to $\gamma$ are parallel to $\gamma$ by Lemma \ref{Lemma: asymptotestogammaareparalleltogamma}, this readily follows.
\end{proof}
\end{Corollary}

Now that we have established the parallelity of $\gamma$ and its asymptotes, we may reparametrise all of them by $\tau$-arclength and fix a shift on each asymptote so that they are synchronised to $\gamma$. Note that we fix a $\tau$-arclength parametrisation of $\gamma$, namely the one keeping $\gamma(0)$ the same.

\begin{definition}[Busemann parametrisation]
\label{definition: busemannparametrization}
Let $\alpha:\R \to X$ be a complete timelike asymptote to $\gamma$. Then we call the (unique) $\tau$-arclength parametrisation of $\alpha$ that synchronises it to $\gamma$ its \emph{Busemann parametrisation}.
\end{definition}

\begin{remark}[Busemann functions]
\label{remark: Busemannfunction}

Up to this point, we have avoided the use of Busemann functions in our treatment. Though it will not really be necessary in what follows, let us discuss them briefly here. Let $X$ be as in the splitting theorem and $\gamma$ the given timelike line in $\tau$-arclength parametrisation. Then we define the (\emph{future}) \emph{Busemann function} of $\gamma$ as $b^+:I(\gamma) \to \R \cup \{\pm \infty\}$ via $b^+(x):=\lim_{t \to \infty} (t - \tau(x,\gamma(t)))$. If $\alpha:\R \to X$ is any asymptotic line to $\gamma$ in $\tau$-arclength parametrisation and we choose the parallel realisation between $\alpha$ and $\gamma$ so that $\gamma(t)$ is mapped to $(t,0) \in \R^{1,1}$ (which we always do) and $\alpha(s)$ is mapped to $(s+S,c) \in \R^{1,1}$, then
\begin{align*}
    b^+(\alpha(s)) = \lim_{t \to \infty}(t - \tau(\alpha(s),\gamma(t)) = \lim_{t \to \infty}(t - \sqrt{(t-(s+S))^2 - c^2}) = s + S.
\end{align*}
Hence, $b^+$ indicates the shift in the time parameter along any $\tau$-arclength parametrised asymptote. In particular, $b^+$ is finite-valued on all of $I(\gamma)$ since any point lies on an asymptote, and in fact it is not hard to show that $b^+$ is continuous on $I(\gamma)$ in a similar fashion as \cite[Lem.\ 3.3]{beem1985toponogov}, but we will not need this. Note that if $\alpha$ is in Busemann parametrisation, then $b^+(\alpha(s)) = s$ for all $s \in \R$, hence the name.
\end{remark}

\begin{Lemma}[Parallelity is weakly transitive]
Let $X$ be a strongly causal, timelike geodesically connected \LpLS with continuous $\tau$, $\tau(x,x)=0$ for all $x$, and global non-negative timelike curvature. Let $\alpha,\beta,\gamma:\R\to X$ be timelike lines such that $\alpha,\beta$ and $\beta,\gamma$ are parallel. Assume that there is a point $p$ on $\gamma$ such that there is a parallel line to $\alpha$ through $p$. Then this parallel line is already $\gamma$.
\end{Lemma}
\begin{proof}
We assume the parallel lines are even synchronised parallel in a suitable $\tau$-arclength parametrisation. Let $\tilde{\gamma}$ be the synchronised parallel line to $\alpha$ through $p$. Let $a$ be the distance between $\alpha,\beta$, $b$ be the distance between $\beta,\gamma$ and $c$ be the distance between $\alpha,\tilde{\gamma}$.
We have two lines through $p$: $\gamma$ and $\tilde{\gamma}$. We assume the parameters for $p$ are $p=\gamma(t_0)=\tilde{\gamma}(\tilde{t}_0)$. By Proposition \ref{Proposition: angle=comparisonangle}, we know that $\tilde{\ma}_p(\gamma(s),\tilde{\gamma}(t))$ is constant in $s$ and $t$ as long as $\gamma(s)\ll\tilde{\gamma}(t)$ (or conversely). We know that $\gamma(s)\ll \beta(s+b+\varepsilon)\ll\alpha(s+a+b+2\varepsilon)\ll\tilde{\gamma}(s+a+b+c+3\varepsilon)$ and set $t=s+a+b+c+3\varepsilon$. We now look at the law of cosines to estimate the comparison angle $\omega=\tilde{\ma}_p(\gamma(s),\tilde{\gamma}(t))$: $\cosh(\omega)=\frac{(s-t_0)^2+(t-\tilde{t}_0)^2-\tau(\gamma(s),\tilde{\gamma}(t))}{2(s-t_0)(t-\tilde{t}_0)}\leq\frac{2(t-\tilde{t}_0)^2}{2(s-t_0)^2}$ which converges to $0$ as $s\to+\infty$. In particular, $\omega=0$. We can now apply the strong causality trick to get $\gamma$ and $\tilde{\gamma}$ are just shifts of each other, so $\alpha$ and $\gamma$ are parallel.

By doing the same argument backwards as well, i.e.\ $\alpha(s-a-b-2\varepsilon)\ll\gamma(s)\ll \alpha(s+a+b+2\varepsilon)$, we get that the synchronised version only differs by a shift of at most $a+b$ (upon taking $\varepsilon \to 0)$. 
\end{proof}

\begin{Lemma}[Verticality in Minkowski space]\label{Lemma: verticalityInMinkowski}
Let $a=(0,0)\ll b$ be points in Minkowski space $\R^{1,1}$. Let the $t$-coordinate difference be $\Delta t=t(b)-t(a)$. Let $c_n$ be points with $b\ll c_n$, the $t$-coordinate $t(c_n)\to+\infty$, $c_n$ above the straight line through $a,b$ and $\tau(a,c_n)-\tau(b,c_n)\to \Delta t$ as $n\to +\infty$. Then $\frac{x(c_n)}{t(c_n)}\to 0$, or equivalently, $\frac{c_n}{\lVert c_n\rVert}\to \partial_t$.\footnote{So $c_n$ converges to the point $[\partial_t]$ on the limit sphere.}
\end{Lemma}
\begin{proof}
We show that $\forall\varepsilon>0$ there are constants $T>0$ and $\delta>0$ such that for all $p$ with $\tau(a,p)\geq T$ and $|\tau(a,p)-\tau(b,p)-(t_b-t_a)|<\delta$, we have that the angle of $ap$ with the $t$-axis satisfies $\ma_a(\partial_t,p)<\varepsilon$.

We set $\omega_1=\ma_a(\partial_t,b)$ the angle between the vertical and $ab$, and $\omega_2=\ma_a(p,b)$ the angle between the vertical and $ap$. We will prove $|\omega_1-\omega_2|<\varepsilon$. By assumption, $p$ and $\partial_t$ lie on the same side of the straight line through $ab$, so by angle additivity in the plane we also get that $\ma_a(\partial_t,p)<\varepsilon$. 

We apply the law of cosines: 
\begin{equation*}
\tau(b,p)^2-\tau(a,p)^2=\tau(a,b)^2-2\tau(a,b)\tau(a,p)\cosh(\omega_2)
\end{equation*}
and use that $\cosh(\omega_1)\tau(a,b)=t_b-t_a$ to get:
\begin{equation*}
\tau(a,p)-\tau(b,p)=(t_b-t_a)\underbrace{\frac{(2\tau(a,p)\cosh(\omega_2)-1)}{\cosh(\omega_1)(\tau(b,p)+\tau(a,p))}}_{=:F}\,.
\end{equation*}
We now have to check that whenever $|\omega_1-\omega_2|\geq\varepsilon$, the factor $F$ is bounded away from $1$.

We have two cases to cover: First, we indirectly assume $\omega_1>\omega_2+\varepsilon$ (then $1-\frac{\cosh(\omega_2)}{\cosh(\omega_1)}>1-\cosh(\varepsilon)$). (Assuming $b$ lies to the right of the $t$-axis, this is the case where $p$ also lies on the right.) We will show $F<1-\delta$ (if $\tau(a,p)$ is large enough and $\tau(a,p),\tau(b,p)$ are close enough). This is equivalent to:
\begin{equation*}
(2\tau(a,p)\cosh(\omega_2)-1)<(1-\delta)\cosh(\omega_1)(\tau(b,p)+\tau(a,p))
\end{equation*}
Dividing by $\tau(a,p)$, this is obvious: $\frac{\tau(b,p)}{\tau(a,p)}\to 1$ as $\tau(a,p)\to+\infty$ and the difference is bounded. We can e.g.\ choose 
\begin{equation*}
\delta<\cosh(\varepsilon)-1
\end{equation*}
and $T$ large enough.

Now we indirectly assume $\omega_1<\omega_2-\varepsilon$ (then $\frac{\cosh(\omega_2)}{\cosh(\omega_1)}-1>\cosh(\varepsilon)-1$) (Assuming $b$ lies to the right of the $t$-axis, this is the case where $p$ lies on the left.) We will show $F>1+\delta$ (if $\tau(a,p)$ is large enough and $\tau(a,p),\tau(b,p)$ are close enough). We need to check:
\begin{equation*}
(2\tau(a,p)\cosh(\omega_2)-1)>(1+\delta)\cosh(\omega_1)(\tau(b,p)+\tau(a,p))
\end{equation*}
Dividing by $\tau(a,p)$, this is obvious: $\frac{\tau(b,p)}{\tau(a,p)}\to 1$ as $\tau(a,p)\to+\infty$ and the difference is bounded. We can e.g.\ choose 
\begin{equation*}
\delta<\cosh(\varepsilon)-1
\end{equation*}
and $T$ large enough.
\end{proof}

We now run into a subtlety: being parallel is in general only weakly transitive, but being synchronised parallel is probably not. Luckily, for $X$ as in the splitting theorem, the fact that we consider asymptotic lines to $\gamma$ simplifies the situation:

\begin{Lemma}[Two asymptotes are parallel] \label{Lemma: twoAsySyncParal}
Let $x,y\in I(\gamma)$ and $\alpha,\beta$ the asymptotes to $\gamma$ through $x$ resp.\ $y$ in Busemann parametrisation. Then $\alpha,\beta$ are synchronised parallel.
\end{Lemma}
\begin{proof}
We dive into the construction of asymptotes: We first assume $x\ll y$. We set $s_0=b^+(x)$ and $t_0=b^+(y)$, then there are maximisers $\alpha_n$ from $x$ to $\gamma(T_n)$ parametrised in $\tau$-arclength such that $\alpha_n(s_0)=x$  and $\beta_n$ from $y$ to $\gamma(T_n)$ parametrised in $\tau$-arclength such that $\beta_n(t_0)=y$, which converge (pointwise) to the upper parts of $\alpha$ and $\beta$ in Busemann parametrisation for $T_n \to \infty$, respectively.
Note that $\alpha_n(s_0+\tau(x,\gamma(T_n)))=\beta_n(t_0+\tau(y,\gamma(T_n)))=\gamma(T_n)$.

We try to prove the $c$-criterion (Lemma \ref{lem-c-cond}): Clearly, due to our assumptions on $X$, we do not need to consider the null functions. We look at $c_{\alpha\beta}(s_0,t_0)$ and $c_{\alpha\beta}(s',t')$ or $c_{\beta\alpha}(s',t')$ for $s_0\leq s'$, $t_0\leq t'$: We require $a:=x=\alpha(s_0)\ll b:=y=\beta(t_0)$. We get points converging to these parameters: $a'_n:=\alpha_n(s')\to a':=\alpha(s')$ and $b'_n:=\beta_n(t')\to b':=\beta(t')$ (note that $\alpha_n(s_0)=a$ and $\beta_n(t_0)=b$ anyway). We get a timelike triangle $\Delta_n=\Delta(a, b, c_n:= \gamma(T_n))$ containing the points $(a'_n,b'_n)$. We form a comparison situation for this in Minkowski space $\R^{1,1}$: $\bar{\Delta}_n=\Delta(\bar{a},\bar{b}, \bar{c}_n)$ with comparison points $\bar{a}'_n,\bar{b}'_n$. As $X$ has global curvature bounded below by $0$, we get that $\tau(a'_n,b'_n)\leq\bar{\tau}(\bar{a}'_n,\bar{b}'_n)$ and $\tau(b'_n,a'_n)\leq\bar{\tau}(\bar{b}'_n,\bar{a}'_n)$.

We would now like to let $T_n \to+\infty$, so we need control over what the comparison triangle $\bar{\Delta}_n$ converges to. For this, we select which way to realise $\bar{\Delta}_n$ in $\R^{1,1}$. We choose:
\begin{itemize}
\item $\bar{a}=(s_0,0)$ constant in $n$, i.e.\ at the right $t$-coordinate and with $x$-coordinate $0$,
\item $\bar{b}=(t_0,c_{\alpha\beta}(s_0,t_0))$ constant in $n$. Note this has the right $\tau$-distance to $\bar{a}$ by definition of $c_{\alpha\beta}$.
\item $\bar{c}_n$ above the straight line $\bar{a}$ and $\bar{b}$ lie on.
\end{itemize}
We claim that in the limit $T_n\to+\infty$, the line $\bar{a}\bar{c}_n$ becomes vertical (the line $\bar{b}\bar{c}_n$ also becomes vertical). For this, we look at what $\bar{\tau}(\bar{a},\bar{c}_n)-\bar{\tau}(\bar{b},\bar{c}_n)$ does as $T_n\to+\infty$: Using Landau small-oh notation, we get that $\bar{\tau}(\bar{a},\bar{c}_n)=T_n-b^+(a) + o(1)$ and $\bar{\tau}(\bar{b},\bar{c}_n)=T_n-b^+(b) + o(1)$ as $T_n\to+\infty$, so the difference is $b^+(b)-b^+(a) + o(1)$, which approaches the $t$-coordinate-difference $t_0-s_0$ of $\bar{a}$ and $\bar{b}$. We can now apply Lemma \ref{Lemma: verticalityInMinkowski} to get that $\frac{x(c_n)}{t(c_n)}\to 0$ and $t(c_n)\to+\infty$, thus these lines become vertical.

We have proven that the comparison points converge, so we have limit comparison points: $\bar{a}=(s_0,0)$ and $\bar{b}=(t_0,c_{\alpha\beta}(s_0,t_0))$ stay fixed anyway, $\bar{a}'_n\to\bar{a}'=(s',0)$, $\bar{b}'_n\to\bar{b}'=(t',c_{\alpha\beta}(s_0,t_0))$. We also set the comparison sides $\bar{\alpha}(s)=(s,0)$ and $\bar{\beta}(t)=(t,c_{\alpha\beta}(s_0,t_0))$. By continuity of $\tau$ and $\bar{\tau}$, our curvature assumption gives $\tau(a',b')\leq\bar{\tau}(\bar{a}',\bar{b}')$ in the limit (similarly with arguments flipped). Now we compare the definition of $c_{\alpha\beta}(s',t')$ resp.\ $c_{\beta\alpha}(s',t')$ with the $c$-functions for $\bar{\alpha}$ and $\bar{\beta}$ at the same parameters:
\begin{align*}
c_{\alpha\beta}(s',t')&=\sqrt{(t'-s')^2-\tau(a',b')^2},\\
c_{\bar{\alpha}\bar{\beta}}(s',t')&=\sqrt{(t'-s')^2-\bar{\tau}(\bar{a}',\bar{b}')^2}=x(\bar{b}')-x(\bar{a}') = c_{\alpha\beta}(s_0,t_0),
\end{align*}
whenever defined. The last equality holds because we know $\bar{\alpha}$ and $\bar{\beta}$ are synchronised parallel with distance $c_{\alpha\beta}(s_0,t_0)$. Note that $a' \leq b'$ implies $\bar{a}' \leq \bar{b}'$ by the curvature bound and a simple continuity argument, so whenever the first line is defined so is the second. Notice these equations only differ in the $\tau$ term, and we know $\tau(a',b')\leq\bar{\tau}(\bar{a}',\bar{b}')$, so we get $c_{\alpha\beta}(s',t')\geq c_{\alpha\beta}(s_0,t_0)$ whenever the former is defined. Similarly, if $b' \leq a'$ we get $c_{\beta\alpha}(s',t')\geq c_{\alpha\beta}(s_0,t_0)$ for all $s'\geq s_0$ and $t'\geq t_0$. We can also do this if $a\gg b$, giving $c_{\alpha\beta}(s',t')\geq c_{\beta\alpha}(s_0,t_0)$ and $c_{\beta\alpha}(s',t')\geq c_{\beta\alpha}(s_0,t_0)$. 

Doing the same construction towards the past, we similarly get $c_{\alpha\beta}(s',t')\geq c_{\alpha\beta}(s_0,t_0)$ and $c_{\beta\alpha}(s',t')\geq c_{\alpha\beta}(s_0,t_0)$ resp.\ $c_{\alpha\beta}(s',t')\geq c_{\beta\alpha}(s_0,t_0)$ and $c_{\beta\alpha}(s',t')\geq c_{\beta\alpha}(s_0,t_0)$ (depending on whether $a\ll b$ or $b\ll a$) for all $s'\leq s_0$ and $t'\leq t_0$.

Now we use Lemma \ref{Lemma: shifting_asymptotes} to get that the (Busemann parametrised) asymptote to $\gamma$ through $\alpha(s)$ is $\alpha$, and similarly the asymptote to $\gamma$ through $\beta(t)$ is $\beta$. In particular, we can use the above argument again for $\alpha(s)$ instead of $x$ and $\beta(t)$ instead of $y$. We know that $b^+(\alpha(s))=s$ and $b^+(\beta(t))=t$. The above (to the future, $a\ll b$) then gives:
$c_{\alpha\beta}(s',t')\geq c_{\alpha\beta}(s,t)$ for $s\leq s'$, $t\leq t'$ such that $a\ll b$ and $a'\ll b'$, extending continuously, $c_{\alpha\beta}$ is monotonously increasing where defined. On the other hand, the above (to the past, $a\ll b$) gives: 
$c_{\alpha\beta}(s',t')\geq c_{\alpha\beta}(s,t)$ for $s\geq s'$, $t\geq t'$ such that $a\ll b$ and $a'\ll b'$, extending continuously, $c_{\alpha\beta}$ is monotonously decreasing where defined. Via a two-step process, we see that $c_{\alpha\beta}$ is constant where defined. Similarly, we get that also $c_{\beta\alpha}$ is constant where defined, and that they have the same value.

Thus, the $c$-criterion (Lemma \ref{lem-c-cond}) yields that $\alpha$ and $\beta$ are synchronised parallel.
\end{proof}

This result establishes the transitivity of synchronised parallel lines: Any two asymptotes to $\gamma$ in Busemann parametrisation are synchronised parallel.

\section{Proof of the main result}\label{sec:proof}

Let us summarise what we have shown so far: Let $(X,d,\ll,\leq,\tau)$ be a connected, regularly localisable, globally hyperbolic Lorentzian length space satisfying timelike geodesic prolongation, with proper metric $d$ and global non-negative timelike curvature containing a complete timelike line $\gamma:\R \to X$. Then from each point in $I(\gamma) = I^+(\gamma) \cap I^-(\gamma)$ we can construct (unique) asymptotic rays to $\gamma$, all of which are timelike with infinite $\tau$-length. Future directed and past directed rays from a common point fit together to give a timelike line which is parallel to $\gamma$. We fix a $\tau$-arclength parametrisation of $\gamma$ fixing $\gamma(0)$, and always consider it to be mapped to the $t$-axis in $\R^{1,1}$ in any parallel realisation. Then, by way of the Busemann parametrisation, all asymptotic lines become synchronised parallel to $\gamma$ and to each other. This synchronisation selects a "spacelike" slice (namely $\{b^+ = 0\}$) in $X$ which will provide the metric part of the splitting.

In this section, we complete the proof of the splitting theorem in two steps: First, we show that $I(\gamma)$ splits, then we show that $I(\gamma)$ has the (TC) property, which implies $I(\gamma) = X$ due to Theorem \ref{Theorem: TCinextendibility}.

\begin{definition}[Spacelike slice]
We call the set of all $\alpha(0)$, where $\alpha$ is a Busemann parametrised timelike asymptotic line to $\gamma$, the \emph{spacelike slice} and denote it by $S$. For $\alpha(0),\beta(0) \in S$, we define $d_S(\alpha(0),\beta(0))$ to be the distance between $\alpha$ and $\beta$ in the sense of parallel lines.
\end{definition}

\begin{Lemma}
$(S,d_S)$ is a metric space.
\end{Lemma}
\begin{proof}
As asymptotes to $\gamma$ are parallel to $\gamma$ and parallel lines are unique, $d_S$ is well-defined.

Let $p,q\in S$. It is obvious from the definition of the distance of two parallel lines that $d_S(p,q)\geq 0$. If $d_S(p,q)=0$, we get that the asymptotes through $p$ and $q$ are the same curve $\alpha$. As the Busemann parametrisation is unique, we get that $p=q$.

For the triangle inequality, let $p,q,r\in S$. We get asymptotes $\alpha$ through $p$, $\beta$ through $q$ and $\gamma$ through $r$, and by Lemma \ref{Lemma: twoAsySyncParal}, they are pairwise synchronised parallel. Let $d_1=d_S(p,q)$ and $d_2=d_S(q,r)$, then $\alpha(0)\leq\beta(d_1)$ by the $c_{\alpha+}^N(0)$ criterion (see Lemma \ref{lem-c-cond}) for $\alpha,\beta$ and $\beta(d_1)\leq\gamma(d_1+d_2)$ by the $c_{\beta+}^N(d_1)$ criterion for $\beta,\gamma$. Thus, we get that $\alpha(0)\leq\gamma(d_1+d_2)$, giving $d_S(p,r)\leq d_1+d_2$ by the $c_{\alpha+}^N(0)$ criterion for $\alpha,\gamma$.
\end{proof}

\begin{definition}[The splitting map]
We define $f:\R\times S\to I(\gamma) \subset X$ by $f(s,p)=\alpha_p(s)$, where $\alpha_p$ is the asymptote to $\gamma$ through $p$ in Busemann parametrisation, that is $\alpha_p(0) = p$.
\end{definition}

\begin{Proposition}[Local splitting]
\label{Theorem: localsplitting}
Let $X$ be a connected, regularly localisable, globally hyperbolic Lorentzian length space with proper metric $d$ and global non-negative timelike curvature satisfying timelike geodesic prolongation and containing a complete timelike line $\gamma:\R \to X$. Then $I(\gamma) \subset X$ is a causally convex open set that is itself a path-connected, regularly localisable, globally hyperbolic Lorentzian length space of global non-negative timelike curvature with the metric, relations and time separation induced from $X$. Moreover, the spacelike slice $S$ is a proper (hence complete), strictly intrinsic metric space, the Lorentzian product $\R \times S$ is a path-connected, regularly localisable, globally hyperbolic Lorentzian length space and the splitting map $f:\R\times S \to I(\gamma)$ is a $\tau$- and $\leq$-preserving homeomorphism.
\end{Proposition}
\begin{proof}
First, it is clear that $I(\gamma)$ is path-connected, causally convex in $X$ and has global non-negative timelike curvature. It is hence causally path-connected since $X$ is and it is trivially locally causally closed. Moreover, if $x \in I(\gamma)$ and $U$ is a regular localising neighbourhood of $x$ in $X$, then $U \cap I(\gamma)$ is a regular localising neighbourhood of $x$ in $I(\gamma)$, hence $I(\gamma)$ is regularly localisable. By causal convexity, the time separation between causally related points in $I(\gamma)$ is achieved as the supremum of lengths of causal curves running between them which have to stay inside $I(\gamma)$. The causal diamonds in $I(\gamma)$ are precisely those in $X$, since they must be contained in $I(\gamma)$. Finally, $I(\gamma)$ is non-totally imprisoning (it inherits this from $X$), thus we have shown all the claims on $I(\gamma)$.

Next, we argue that $f$ is $\tau$- and $\leq$-preserving: Let $p,q\in S$. Let $\alpha$ be the asymptote to $\gamma$ through $p$ and let $\beta$ be the asymptote to $\gamma$ through $q$. Then $\alpha$ and $\beta$ are synchronised parallel with distance $d_S(p,q)$, so there is a parallel realisation $\tilde{f}:\alpha(\R)\cup\beta(\R)\to\R^{1,1}$ defined by $\tilde{f}(\alpha(s))=(s,0)$ and $\tilde{f}(\beta(t))=(t,d_S(p,q))$. In particular, $\alpha(s)\leq_X \beta(t)\Leftrightarrow t-s\geq d_S(p,q)$ and if this is true, $\tau(\alpha(s),\beta(t))=\sqrt{(t-s)^2-d_S(p,q)^2}$. But that is just the definition of $(s,p)\leq_{\R\times S}(t,q)$ resp.\ $\tau_{\R\times S}((s,p),(t,q))$ in $\R\times S$. This is true for all $p,q\in S$ and $s,t\in \R$, so $f$ is $\tau$- and $\leq$-preserving. $f$ is injective as asymptotes to $\gamma$ are parallel to $\gamma$ and parallel lines are unique, and surjective since any point in $I(\gamma)$ lies on an asymptote.

From the discussion above, it is easy to see that $f$ maps timelike (and causal) diamonds in $I(\gamma)$ to timelike (and causal) diamonds in $\R \times S$. As both sides are strongly causal, this implies that $f$ is a continuous open bijection, hence a homeomorphism. Since $\R \times S$ is always non-totally imprisoning (cf.\ Proposition \ref{proposition: nontotalimprisonmentproducts}) and its causal diamonds are compact (as continuous images of compact causal diamonds in $I(\gamma)$), we conclude that $\R \times S$ is globally hyperbolic, hence $S$ is proper by Proposition \ref{Proposition: productglobhypiffmetricproper}. To see that $S$ is a strictly intrinsic space, fix $p,q \in S$ and connect any two timelike related points on the corresponding asymptotes by a distance realiser in $I(\gamma)$. The image of that distance realiser under $f$ is a continuous distance realiser in $\R \times S$, hence by Proposition \ref{Proposition: distancerealizersinproducts} the projection onto $S$ gives a distance minimiser in $S$ between $p$ and $q$.

Finally, we need to prove the remaining claimed properties of $\R \times S$. Path-connectedness is inherited from $I(\gamma)$ via $f$, and products are always globally causally closed (cf.\ Proposition \ref{Proposition: Lorproductscontintaucausclosed}). Note that $I(x,y)$ for $x,y \in I(\gamma)$ are regular localising neighbourhoods in $I(\gamma)$. Since $f(I(x,y)) = I(f(x),f(y))$ and the $d$-lengths of causal curves in timelike diamonds can always be uniformly bounded in products (cf.\ the proof of Proposition \ref{proposition: nontotalimprisonmentproducts}), timelike diamonds in $\R \times S$ are in fact (regular) localising neighbourhoods: For the local time separation, take the restriction of $\tau_{\R \times S}$, and note that maximisers in $I(\gamma)$ (or in any $I(x,y) \subset I(\gamma)$) map to continuous maximisers in $\R \times S$, which are always Lipschitz reparametrisable and are hence causal curves (cf.\ Proposition \ref{Proposition: distancerealizersinproducts}).
\end{proof}

\begin{Proposition}[Global splitting]
\label{Theorem: globalsplitting}
Let $(X,d,\ll,\leq,\tau)$ satisfy the assumptions in Proposition \ref{Theorem: localsplitting}. Then $I(\gamma) = X$, i.e.\ the splitting in Proposition \ref{Theorem: localsplitting} is global.
\end{Proposition}
\begin{proof}

Suppose $\alpha:[a,b) \to I(\gamma)$ is a future directed timelike geodesic with finite $\tau$-length. We reparametrise $\alpha$ by $\tau$-arclength and denote the reparametrised curve again by $\alpha$, it is then defined on $[0,L)$, where $L:=L_{\tau}(\alpha) < \infty$. Then $f^{-1}\circ\alpha$ is a continuous timelike geodesic in $\R \times S$, where $f$ is the splitting map. We decompose it into time- and space-component: $f^{-1}\circ\alpha=(\alpha^t,\alpha^x)$. If $\alpha^x$ is constantly $p$, we see that $\alpha=\alpha_p$, a contradiction since asymptotes have infinite $\tau$-length. Note that since the $\tau$-lengths of $f^{-1} \circ \alpha$ and $\alpha$ agree, $f^{-1} \circ \alpha$ also has finite $\tau$-length in $\R \times S$. By Corollary \ref{Corollary: tauarclengthofmaxinproducts}, $\alpha^x$ has constant speed and domain $[0,L)$, thus finite $d_S$-length.

$\alpha^t$ can certainly be extended to $L$. As $S$ is a complete metric space, we can extend $\alpha^x$ continuously to $L$. Mapping this extension back to $I(\gamma)$ via $f$, we get a continuous extension in $I(\gamma)$ of $\alpha$ to $b$ (in its original parametrisation). This shows that $I(\gamma)$ has the (TC) property, thus $I(\gamma) = X$ by Theorem \ref{Theorem: TCinextendibility}. This concludes the proof of the splitting in Theorem \ref{Theorem: Truesplitting}.
\end{proof}

Recall the notion of Cauchy sets and (Cauchy) time functions on Lorentzian pre-length spaces from \cite[Sec.\ 5.1]{burtscher2021time}: A Cauchy set is any subset that is met exactly once by doubly inextendible causal curves, and a Cauchy time function is a continuous function $t:X \to \R$ such that $x < y$ implies $t(x) < t(y)$ and the image under $t$ of any doubly inextendible causal curve is all of $\R$.

\begin{Corollary} Let $(X,d,\ll,\leq,\tau)$ satisfy the assumptions of Theorem \ref{Theorem: Truesplitting} and let $f:\R \times S \to X$ be the splitting. Then the sets $S_t:=f(\{t\} \times S)$ are Cauchy sets in $X$ that are all homeomorphic to $S$. Moreover, the map $pr_1 \circ f^{-1}$ is a Cauchy time function. Moreover, all Cauchy sets in $X$ are homeomorphic to $S$.
\begin{proof}
Let $\alpha:(a,b) \to X$ be any doubly inextendible causal curve in $X$ meeting any $S_t$ twice, then its image under $f^{-1}$ meets $\{t\} \times S$ twice, which cannot happen since $\{t\} \times S$ is acausal in $\R \times S$. Next we argue that any such $\alpha$ meets $S_t$: For simplicity let $t=0$, and suppose that $\alpha$ does not meet $S$, so w.l.o.g.\ we may assume that $\alpha \subset I^+(S)$ (this is because $X=I^-(S) \cup S \cup I^+(S)$ due to the splitting, and this union is disjoint). Let $t_0 \in (a,b)$, then $\alpha((a,t_0]) \subset J^-(\alpha(t_0)) \cap J^+(S)$ which is easily seen to be a compact set by considering the corresponding situation in $\R \times S$. But this is a contradiction, since $X$ is non-totally imprisoning. Hence $S$ (and any $S_t$) is a Cauchy set. They are homeomorphic to $S$ since $X$ has the topology of $\R \times S$.

We now argue that $pr_1 \circ f^{-1}$ is a Cauchy time function for $X$. It is clearly a time function, as it is continuous and the time separation on $\R \times S$ is more or less a reformulation of that. Now let $\alpha:(a,b) \to X$ be a doubly inextendible causal curve, we need to show that $pr_1 \circ f^{-1} \circ \alpha((a,b)) = \R$. Suppose not, so there is some time value $t_0$ that is not attained. W.l.o.g.\ suppose $t_0 \geq 0$, so the image is contained in $(-\infty,t_0]$. Similarly to before, this would imply that $\alpha|_{[t_0,b)}$ is contained in the compact set $J^+(\alpha(t_0)) \cap J^-(S_{t_0})$, a contradiction.

Finally, let $C$ be any Cauchy set in $X$. Then the projection $C \to S$ (via the splitting) is continuous. Its inverse is given by sending each $p \in S$ to the unique point on $\alpha_p$ meeting $C$. This is a continuous map since $C$ is achronal. This shows that $C$ is homeomorphic to $S$.
\end{proof}
\end{Corollary}

Note that in general, Cauchy sets in globally hyperbolic Lorentzian pre-length spaces need not be homeomorphic, as \cite[Ex.\ 5.7,\ Ex.\ 5.8]{burtscher2021time} show.

\begin{Corollary}
Let $(X,d,\ll,\leq,\tau)$ satisfy the assumptions of Theorem \ref{Theorem: Truesplitting}. Then $(S,d_S)$ has Alexandrov curvature $\geq 0$.
\begin{proof}
Due to the splitting, the space $\R \times S$ has global nonnegative timelike curvature (cf.\ Remark \ref{remark: contvsLipschitztriangles}; the splitting maps Lipschitz triangles in $\R \times S$ to continuous triangles in $X$, but this does not lead to any issues), and then it follows from \cite[Thm.\ 5.7]{alexander2019generalized} that $S$ has Alexandrov curvature $\geq 0$.
\end{proof}
\end{Corollary}

\section{Applications and Outlook}
\label{sec: applicationsandoutlook}

Let us now give several reformulations of our splitting result for spacetimes. We start with smooth spacetimes, where we get a weaker version of the smooth splitting result proven in \cite{beem1985toponogov}.

\begin{Corollary}[Smooth spacetimes]
\label{Corollary: synthsplitsmoothspacetimes}
Let $(M,g)$ be a smooth, globally hyperbolic spacetime with non-positive timelike sectional curvature containing a complete timelike line $\gamma:\R \to M$. Then the Lorentzian length space induced by $(M,g)$ splits in the sense of Theorem \ref{Theorem: Truesplitting}.
\begin{proof}
Fix any complete Riemannian metric $h$ on $M$ and let $d_h$ be the induced distance, then $(M,d_h,\ll,\leq,\tau)$ is a connected, regularly localisable (via convex neighbourhoods) Lorentzian length space and $d_h$ is a proper metric (note that the $\tau$-length of causal curves agrees with the usual notion of length in $(M,g)$ due to \cite[Prop.\ 2.32]{kunzinger2018lorentzian}). The fact that $(M,g)$ has global non-negative timelike curvature in the sense of Definition \ref{definition: globalcurvbounds} is a consequence of the Lorentzian Toponogov theorem (see \cite[p.\ 303]{harris1982triangle}). The remaining technical assumptions are easily seen to hold. Hence the Lorentzian length space induced by $(M,g)$ can be split according to Theorem \ref{Theorem: Truesplitting}.
\end{proof}
\end{Corollary}

\begin{remark}
In the setting of Corollary \ref{Corollary: synthsplitsmoothspacetimes}, once one has shown that the spacelike slice $S$ is a spacelike hypersurface in $M$ and the normal asymptotic distance $d_S$ is the Riemannian distance on $S$ (e.g.\ via an analysis of the gradient of the Busemann function as in \cite{beem1985toponogov}), the Hawking--King--McCarthy theorem \cite{hawking1976new}, which states that any $\tau$-preserving homeomorphism between strongly causal spacetimes is a smooth isometry, provides a somewhat alternative proof of the smooth splitting result proven in \cite{beem1985toponogov}.
\end{remark}

Next, we turn to $C^{1,1}$-spacetimes. First note that the Lorentzian pre-length space induced by any strongly causal Lipschitz spacetime is in fact a Lorentzian length space by \cite[Thm.\ 5.12]{kunzinger2018lorentzian}. For $C^{1,1}$-metrics, the notion of sectional curvature is not available any more, so we have to assume synthetic curvature bounds. However, most of the other technical assumptions are easily seen to be satisfied due to the existence of convex neighbourhoods. For details on low regularity spacetimes, see e.g.\ the recent review article \cite{steinbauer2022singularity} and the references therein.

\begin{Corollary}[$C^{1,1}$-spacetimes]
Let $(M,g)$ be a globally hyperbolic $C^{1,1}$-spacetime with global non-negative timelike curvature containing a complete timelike line $\gamma:\R \to M$. Then the Lorentzian length space induced by $(M,g)$ splits in the sense of Theorem \ref{Theorem: Truesplitting}.
\end{Corollary}

Finally let us mention the case of $C^1$-spacetimes, which are substantially different. Notably, the Lorentzian length space induced by a globally hyperbolic $C^1$-spacetime need not satisfy the timelike geodesic prolongation property. This is because solutions of the $g$-geodesic equation need not maximise locally, nor does a maximising segment necessarily extend as a local maximiser beyond its endpoints. However, points in $C^1$-spacetimes still have neighbourhood bases consisting of globally hyperbolic, causally convex neighbourhoods, so regular localisability is satisfied. Assuming the necessary synthetic properties as well, we have the following (note that geodesics in the synthetic sense are locally maximising curves, which are necessarily $C^2$-solutions of the geodesic equation even if the metric is just $C^1$):

\begin{Corollary}[$C^1$-spacetimes]
Let $(M,g)$ be a globally hyperbolic $C^1$-spacetime with global non-negative timelike curvature containing a complete timelike line $\gamma:\R \to M$. Suppose that each maximising timelike geodesic segment extends to a locally maximising timelike geodesic on an open domain. Then the Lorentzian length space induced by $(M,g)$ splits in the sense of Theorem \ref{Theorem: Truesplitting}.
\end{Corollary}

In this work, we have established a synthetic splitting result for Lorentzian length spaces using triangle comparison. Ideas for future projects include the generalisation to Lorentzian length spaces lower on the causal ladder (e.g.\ (TC) spaces instead of globally hyperbolic ones, an approach to this that seems to be adaptable to the synthetic situation is \cite{galloway1996regularity}), an analysis of the aforementioned splittings of low regularity spacetimes (regarding regularity of the splitting, level set, etc.), as well as an investigation into synthetic timelike Ricci curvature bounds introduced in \cite{cavalletti2020optimal} and expanded in \cite{braun2022r}. In the case of low regularity spacetimes, it is not yet clear how these different notions of curvature bounds relate to each other (some recent progress in this direction includes \cite{kunzinger2022synthetic,braun2022timelike}). For metric measure spaces, it is known that the $\mathrm{CD}$-condition is not sufficient and the full $\mathrm{RCD}$-condition is needed for a splitting result (see \cite{gigli2013splitting, gigli2014overview}). An analogue of the $\mathrm{RCD}$-condition is currently unavailable for Lorentzian length spaces and will probably have to be developed before a splitting theorem with synthetic timelike Ricci curvature bounds can be considered.\\

\noindent{\em Acknowledgements.} Argam Ohanyan and Felix Rott were supported by project P 33594 of the Austrian Science Fund FWF. Argam Ohanyan was also supported by the ÖAW-DOC scholarship of the Austrian Academy of Sciences. \ Didier Solis acknowledges the support of Conacyt under grant SNI 38368 and UADY under program FMAT-PTA2022, as well as the hospitality of the University of Vienna, where parts of this work were conducted. Tobias Beran acknowledges the support of University of Vienna. The authors would like to thank Matteo Calisti, Gregory Galloway, Melanie Graf, Michael Kunzinger, Clemens Sämann, Benedict Schinnerl and Roland Steinbauer for helpful discussions and valuable feedback. 

%\bibliography{Bibliography} 
\bibliographystyle{abbrv}
%\bibliographystyle{alpha}
%\end{document}

\end{document}